\documentclass[12pt,letterpaper]{article}
\usepackage{amssymb,amsfonts,color,graphicx, amsmath,mathrsfs,mathrsfs,amsthm,theoremref,natbib}
\usepackage[english]{babel}
\usepackage{hyperref}

\usepackage[abbrev, nobysame]{amsrefs}
\newtheorem{theorem}{Theorem}[section]
\newtheorem{rem}{Remark}[section]

\newtheorem{cor}{Corollary}[section]
\newtheorem{prop}{Proposition}[section]

\numberwithin{equation}{section}

\newcommand{\R}{\mathbb{R}}

\newcommand{\C}{\mathbb{C}}

\def \p {\partial}

\title{Partial data inverse problems for magnetic Schr\"odinger operators with potentials of low regularity}

\author{Salem Selim \footnote{Department of Mathematics, University of California, Irvine, CA 92697-3875, USA, {\sf selimsa@uci.edu}}}
\date{}

\begin{document}
\bibliographystyle{amsplain}
\maketitle

\begin{abstract}
We establish a global uniqueness result for an inverse boundary problem with partial data for the magnetic Schr\"odinger operator with a magnetic potential of class $W^{1,n}\cap L^\infty$, and an electric potential of class $L^n$. Our result is an extension, in terms of the regularity of the potentials, of the results \cite{ferreira2006determining} and \cite{KnudsenSalo}. As a consequence, we also show global uniqueness for a partial data inverse boundary problem for the advection--diffusion operator with the advection term of class $W^{1,n}\cap L^\infty$.  
\end{abstract}
\tableofcontents

\section{Introduction and main results}

The purpose of this paper is to study inverse boundary problems with partial data for magnetic Schr\"odinger operators, as well as advection--diffusion operators, with potentials of low regularity.  

Let us start by introducing the problem under consideration in the geometric setting of compact Riemannian manifolds with boundary. To that end, let $(M,g)$ be a smooth compact Riemannian manifold of dimension $n\geq 3$ with smooth boundary $\partial M$. Let $d:C^\infty(M) \to C^\infty(M,T^*M)$ be the de Rham differential, and let $A \in C^\infty(M,T^*M)$ be a $1$--form with complex valued $C^\infty$ coefficients. Then we introduce 
\[
d_A := d +i A : C^\infty(M) \to C^\infty(M, T^*M).
\]
The formal  $L^2$--adjoint $d_A^* : C^\infty(M,T^*M) \to C^\infty(M)$ of $d_A$ is given by 
\[
(d_A u , v)_{L^2(M,T^*M)} = (u,d_A^*v)_{L^2(M)}, \hspace{2mm} u\in C_0^\infty(M^0), \hspace{2mm} v\in C_0^\infty(M^0, T^*M^0),
\]
where $M^0=M\setminus \p M$ stands for the interior of $M$, $(\cdot,\cdot)_{L^2(M)}$ is the $L^2$--scalar product of functions on $M$, and $(\cdot , \cdot)_{L^2(M,T^*M)}$ is the $L^2$--scalar product in the space of $1$--forms on $M$, given by 
\[
(\alpha , \beta)_{L^2(M,T^*M)}=\int_M \langle \alpha , \overline{\beta}\rangle_g dV_g, \quad \alpha,\beta\in C^\infty(M,T^*M). 
\]
Here $\langle \cdot , \cdot \rangle_g$ is the pointwise scalar product in the space of $1$--forms induced by the Riemannian metric $g$, and $dV_g$ is the Riemannian volume element on $M$. In local coordinates $(x^1,\dots, x^n)$ in which $\alpha=\alpha_j dx^j$, $\beta=\beta_j dx^j$ and $(g^{jk})$ is the matrix inverse of 
$(g_{jk})$, $g=g_{jk}dx^jdx^k$, we have
\[
\langle \alpha, \beta\rangle_g=g^{jk}\alpha_j \beta_k. 
\]
Here and in what follows we use Einstein's summation convention. We shall denote by $d^*$ the formal  $L^2$--adjoint of $d$, which in local coordinates is given by 
\begin{equation}
\label{eq_int_1}
d^*v=-|g|^{-\frac{1}{2}}\p_{x_j}(|g|^{\frac{1}{2}}g^{jk}v_k),
\end{equation}
where $|g|=\det(g_{jk})$ and $v=v_jdx^j$.  We shall also need the following expression for $d_A^*$ 
\begin{equation}
\label{eq_int_2}
d_A^*=d^*-i \langle\overline{A}, \cdot\rangle_g.
\end{equation}
From now on let us assume that $A \in (W^{1,n}\cap L^\infty)(M^0,T^*M^0)$ and let $q \in L^n(M,\mathbb{C})$. 
We define the magnetic Schr\"odinger operator by
\begin{equation}
    \label{1.1}
\begin{aligned}
  L_{A,q} u= (d^*_{\overline{A}} d_A+q)u&= -\Delta_g u + id^*(Au) -i \langle A,du \rangle_g + (\langle A,A\rangle_g +q)u\\
  &=-\Delta_g u + i(d^*A)u -2i \langle A,du \rangle_g + (\langle A,A\rangle_g +q)u,
\end{aligned}
\end{equation}
where $u \in H^1(M^0)$. Note that in the second equality in \eqref{1.1} we used \eqref{eq_int_2} while in the last equality we used that 
\begin{equation}
\label{eq_1.1-useful_mag}
d^*(Au)=(d^*A)u -\langle A,du \rangle_g. 
\end{equation}
Let $u \in H^1(M^0)$ be such that 
\begin{equation}
\label{1.2_main}
L_{A,q} u = 0\quad \text{in}\quad \mathcal{D}'(M^0).
\end{equation}
It follows from \eqref{1.1} that $\Delta_g u\in L^2(M)$, and thus the boundary trace $\p_\nu u|_{\p M}\in H^{-\frac{1}{2}}(\p M)$ is well defined, see \cite{Bukhgeim_Uhlmann_2002}. Here and in what follows $\nu$ is the unit outer normal to the boundary of $M$. Proposition \ref{prop_trace_magnetic} in Section \ref{sec_recovery_magnetic} gives that $(\langle A,\nu\rangle_g u)|_{\p M}\in H^{-\frac{1}{2}}(\p M)$. 

Let $\Gamma\subset \p M$ be open non-empty. Associated to $\Gamma$, let us introduce the set of partial Cauchy data for solutions of the magnetic Schr\"odinger equation \eqref{1.2_main},  
\begin{equation}
    \label{1.3}
\begin{aligned}
    C_{A,q}^\Gamma := \{ (u|_{\partial M}, (\p_\nu u+i \langle A,\nu\rangle_g u)|_\Gamma) : u \in H^1(M^0)  \text{ satisfies } L_{A,q}u=0 \text{ in } \mathcal{D}'(M^0) \}\\
    \subset H^{\frac{1}{2}}(\p M)\times H^{-\frac{1}{2}}(\Gamma).
\end{aligned}
\end{equation}

The inverse problem that we are interested in concerns the recovery of the magnetic field $dA$ and the electric potential $q$ in $M$ from the knowledge of the set $C_{A,q}^\Gamma$.  To state our results, we need to introduce some assumptions on the manifold $M$ and the open set $\Gamma$. 

To this end, following \cite{DKSaU_2009}, we assume that the manifold $M$ is admissible,  in the sense that there exists an $(n-1)$-dimensional compact simple manifold $(M_0,g_0)$ such that $M \subset \mathbb{R}\times M_0^0$, with $g = c(e \oplus g_0)$, where $e$ is the Euclidean metric on $\mathbb{R}$ and $0 < c\in C^\infty(M)$. Here the manifold $(M_0,g_0)$ is simple if for any $p \in M_0$, the exponential map $\exp_p$ with its maximal domain of definition in $T_p M_0$ is a diffeomorphism onto $M_0$, and if $\partial M_0$ is strictly convex in the sense that the second fundamental form of $\partial M_0 \hookrightarrow M_0$ is positive definite.

Let $\varphi\in C^\infty(\R\times M_0)$ be given by $\varphi(x) = x_1$, $x=(x_1,x')\in \R\times M_0$. The significance of this function is that it is a limiting Carleman weight on $M$ in the sense of \cite{Kenig_Sjostrand_Uhlmann_2007}, \cite{ferreira2006determining}, \cite{DKSaU_2009}.  We introduce the front side of $\partial M$ as follows,  
\begin{equation}
\label{eq_int_p_M-}
F = \{x \in \partial M : \partial_\nu \varphi(x) \le 0 \}. 
\end{equation}
Our main result is as follows. 
\begin{theorem}
\label{thm1.1}
Let $(M,g)$ be an admissible simply connected manifold of dimension $n\geq 3$ with connected boundary. Let $A_1,A_2 \in (W^{1,n}\cap L^\infty)(M^0,T^*M^0)$ be complex valued 1-forms such that $A_1|_{\p M} = A_2|_{\p M}$ and $q_1,q_2 \in L^n(M,\mathbb{C})$.  Let $\Gamma \subset \partial M$ be an open neighborhood of $F$. If  $C_{A_1,q_1}^{\Gamma}= C_{A_2,q_2}^{\Gamma}$ then  $dA_1 = dA_2$ and $q_1 = q_2$ in $M$. 
\end{theorem}
\noindent In Theorem \ref{thm1.1} we have $A_j|_{\p M}\in L^p(\p M)$ for all $1\le p<\infty$, $j=1,2$, see Proposition \ref{prop_trace_magnetic_2} in Section \ref{sec_recovery_magnetic} below.

As a consequence of Theorem \ref{thm1.1}, we obtain a global uniqueness result for the partial data inverse problem for the magnetic Schr\"odinger operator on a domain in the Euclidean space. To state the result, let $\Omega\subset \R^n$, $n\ge 3$, be open bounded with $\p \Omega\in C^\infty$. When $A\in (W^{1,n}\cap L^\infty)(\Omega, \C^n)$ and $q\in L^n(\Omega,\C)$, we consider the magnetic Schr\"odinger operator on $\Omega$ given by 
\begin{equation}
    \label{1.1_eclid}
L_{A,q}u = \sum_{j=1}^n(D_{x_j}+A_j)^2u+qu= -\Delta u+ A\cdot Du + D\cdot (Au) + (A^2+q)u,
\end{equation}
where $D = -i \nabla$, and $u\in H^1(\Omega)$. Let $x_0\in \R^n\setminus\overline{{\rm ch}(\Omega)}$, and let $\varphi(x)=\log|x-x_0|$. Here ${\rm ch}(\Omega)$ is the convex hull of $\Omega$. In analogy with \eqref{eq_int_p_M-}, we let 
\begin{equation}
\label{eq_int_p_M-_eclid}
F(x_0) = \{x \in \partial M : \partial_\nu \varphi(x) \le 0 \}.
\end{equation}

\begin{theorem}
\label{thm1.2}
Let $\Omega\subset \R^n$, $n\ge 3$, be open bounded simply connected with connected  $C^\infty$ boundary. Let 
$A_1, A_2\in (W^{1,n}\cap L^\infty)(\Omega, \C^n)$ such that $A_1|_{\p \Omega} = A_2|_{\p \Omega}$ and $q_1,q_2\in L^n(\Omega,\C)$. Let $\Gamma \subset \partial M$ be an open neighborhood of $F(x_0)$. Assume that $C_{A_1,q_1}^{\Gamma}= C_{A_2,q_2}^{\Gamma}$. We have $dA_1 = dA_2$ and $q_1 = q_2$ in $\Omega$. 
\end{theorem}

\begin{rem}
Theorem  \ref{thm1.1} and Theorem  \ref{thm1.2} can be viewed as extensions of the partial data results of 
\cite{ferreira2006determining}, where the case of $A\in C^2$ and $q\in L^\infty$ was considered, and the result of \cite{KnudsenSalo}, where $A\in C^\varepsilon$, $\varepsilon>0$, and $q\in L^\infty$. 
\end{rem}

\begin{rem}
To the best of our knowledge, Theorem  \ref{thm1.1} and Theorem  \ref{thm1.2} are first partial data results for magnetic Schr\"odinger operators with possibly discontinuous magnetic potentials. 
\end{rem}

Let us next consider an application of Theorem \ref{thm1.1} to partial data inverse problems for  advection-diffusion equations. To this end, given a real vector field $X \in (W^{1,n}\cap L^\infty) (M^0,TM^0)$, we introduce the following Dirichlet problem 
\begin{equation}
\label{eq1.4}
    \begin{cases}
    L_X u := (-\Delta_g +X)u =0 \hspace{3mm} \text{ in } \mathcal{D}'(M^0),\\
    u|_{\partial M} = f \in H^{1/2}(\partial M)
    \end{cases}
\end{equation}
which by \cite{Aubin}*{Chapter 3, Section 8.2} has a unique solution $u \in H^1(M^0)$. Given an open neighborhood $\Gamma$ of $F$  introduced in \eqref{eq_int_p_M-}, we define the partial Dirichlet--to--Neumann map by
\begin{equation}
    \label{eq1.5}
    \Lambda_X^\Gamma : H^{1/2}(\partial M) \to H^{-1/2}(\Gamma), \hspace{5mm} f \to \partial_\nu u|_{\Gamma}.
\end{equation}
We have the following result. 
\begin{theorem}
    \label{thm1.3}
   Let $(M,g)$ be an admissible simply connected manifold of dimension $n\geq 3$ with connected boundary.   Let $X_1,X_2\in (W^{1,n}\cap L^\infty) (M^0, TM^0)$ be real vector fields such that $X_1|_{\p M} = X_2|_{\p M}$. Let $\Gamma \subset \partial M$ be an open neighborhood of $F$ defined in \eqref{eq_int_p_M-}. If  $\Lambda_{X_1}^\Gamma = \Lambda_{X_2}^\Gamma$ then  $X_1 = X_2$ in M.
\end{theorem}

Specializing Theorem \ref{thm1.3} to the case of domains in $\R^n$, we get the following result. 
\begin{theorem}
    \label{thm1.4}
    Let $\Omega\subset \R^n$, $n\ge 3$, be open bounded simply connected with connected  $C^\infty$ boundary, and let 
$X_1, X_2\in (W^{1,n}\cap L^\infty)(\Omega, \R^n)$ such that $X_1|_{\p \Omega} = X_2|_{\p \Omega}$. Let $\varphi(x)=\log|x-x_0|$, where $x_0\in \R^n\setminus\overline{ch(\Omega)}$. Let $\Gamma \subset  \p \Omega$ be an open neighborhood of $F(x_0)$, given in \eqref{eq_int_p_M-_eclid}. If  $\Lambda_{X_1}^\Gamma = \Lambda_{X_2}^\Gamma$ then  $X_1 = X_2$ in $\Omega$.    
\end{theorem}

\begin{rem}
Theorem  \ref{thm1.3} and Theorem  \ref{thm1.4} can be viewed as extensions of the partial data result of \cite{KnudsenSalo}, where the advection term satisfies $X\in C^\varepsilon$, $\varepsilon>0$, and $\nabla\cdot X\in L^\infty$. 
\end{rem}

The study of inverse boundary problems for magnetic Schr\"odinger operators has a long tradition in inverse problems. Let us proceed to recall some of the fundamental contributions in dimension $n \geq 3$, first in the full data case, for domains in the Euclidean space. Following the fundamental paper \cite{Sylvester_Uhlmann_1987} when $A=0$, a global uniqueness result in the presence of a $C^\infty$ magnetic potential was established in \cite{Nakamura1995GlobalIF}, see also \cite{sun}. Regularity assumptions on the magnetic potential were subsequently weakened in the works  \cite{Tolmasky},  \cite{Panchenko_2002}, \cite{Salo2004InversePF}. The sharpest currently available result in terms of the regularity of the magnetic potential was obtained in \cite{UhlmannKatya2014} for $A\in L^\infty$. We refer to \cite{Haberman_2018} for an extension of this result in the three dimensional case, under a suitable smallness assumption, and to \cite{Chanillo_1990}, \cite{Lavine_Nachman} for the case when $A=0$, $q\in L^{\frac{n}{2}}$. Turning the attention to the setting of admissible manifolds as well as its generalization when the transversal manifold is no longer simple, still in the full data case, we refer to  the works \cite{DKSaU_2009}, \cite{unboundedPotential}, \cite{DSKurLassSalo_2016}, \cite{Cekic_2017}, \cite{2018}. In particular, the paper \cite{2018} showed a global uniqueness result for $A\in L^\infty$ in the setting of admissible manifolds. 

The study of partial data inverse problems for Schr\"odinger operators in dimension $n\geq 3$, for domains in the Euclidean space, was initiated in the pioneering works \cite{Bukhgeim_Uhlmann_2002}, \cite{Kenig_Sjostrand_Uhlmann_2007} when $A=0$ and $q\in L^\infty$. We refer to \cite{Chung_Tzou_2020}, \cite{Tzou_preprint} for extensions of these results to $q\in L^{\frac{n}{2}}$, while still $A=0$. The magnetic case was then treated in \cite{ferreira2006determining} for $A\in C^2$ and $q\in L^\infty$,  see \cite{PMRuizT_2022} for the corresponding stability result, and see also \cite{FrancisChung}, \cite{Chung_2014}.  In \cite{KnudsenSalo}, the regularity of magnetic potentials was relaxed to the H\"older continuity. In the context of admissible manifolds, partial data inverse problems were studied in \cite{Kenig_Salo_2013}, \cite{Bhattacharyya_2018}.  We refer to the survey papers \cite{Uhlmann_2014_review} and \cite{Kenig_Salo_2014_review} for a fuller account of the work done on partial data inverse problems and for additional references. 

Inverse boundary problems for advection--diffusion equations in the full data case were studied in \cite{Salo2004InversePF}, \cite{Cheng_Nakamura_Somersalo_2001}, \cite{Pohjola_2015}, for domains in the Euclidean space, and in  \cite{krupchykManifold} for admissible manifolds. An extension of the partial data result of \cite{ferreira2006determining} to the case of advection--diffusion equations was obtained in \cite{KnudsenSalo} for the advection term satisfying the conditions $X\in C^\varepsilon$, $\varepsilon>0$, and $\nabla\cdot X\in L^\infty$. 

Let us proceed to discuss the main ideas in the proofs of our results, starting with Theorem  \ref{thm1.1}. Following the tradition of works on partial data inverse boundary problems, going back to \cite{Bukhgeim_Uhlmann_2002}, \cite{Kenig_Sjostrand_Uhlmann_2007},  \cite{ferreira2006determining}, to establish Theorem  \ref{thm1.1}, we rely on complex geometric  optics (CGO) solutions to the magnetic Schr\"odinger equations, as well as on Carleman estimates with boundary terms for the magnetic Schr\"odinger operator.  To construct CGO solutions to the magnetic Schr\"odinger equation $L_{A,q}u=0$ with $A \in (W^{1,n}\cap L^\infty)(M^0,T^*M^0)$ and $q \in L^n(M,\mathbb{C})$, with sufficient decay of the remainder term, we make use of a smoothing argument, approximating $A$ by a sequence of smooth one forms $A_\tau$ such that 
\[
\|A-A_\tau\|_{L^n}=o(\tau), \quad \|\nabla_g^2A_\tau^\sharp\|_{L^n}=o(\tau^{-1}),
\]
as $\tau\to 0$. Here $A^\sharp_\tau=\sum_{j,k=1}^n g^{jk}(A_\tau)_k\p_{x_j}$ is the vector field, corresponding to the one form $A_\tau$. Working with the limiting Carleman weight $\varphi(x) = x_1$, the recovery of the magnetic field and electric potential is performed by making use of results related to the injectivity of the attenuated ray transform, as in the full data case discussed in the works \cite{DKSaU_2009}, \cite{unboundedPotential}, \cite{2018}. 

To establish the partial data result with logarithmic weights of Theorem \ref{thm1.2}, rather than relying on techniques of analytic microlocal analysis as in the works \cite{Kenig_Sjostrand_Uhlmann_2007}, \cite{ferreira2006determining}, \cite{KnudsenSalo}, we make use of a change of variables  as in \cite{Kenig_Salo_2013}, \cite{caseyRodriguez},  allowing us to view the domain $\Omega$ as an admissible manifold, with the logarithmic weight becoming a linear one in the new coordinates. Theorem \ref{thm1.2} becomes then a direct consequence of Theorem  \ref{thm1.1}. 

To be on par with the best available full data result, one would next like to try to establish a partial data result for magnetic potentials of class $L^\infty$ only. Let us point out one of the main challenges in dealing with such a problem. In our case, when working with $A\in (W^{1,n}\cap L^\infty)(M^0,T^*M^0)$ and $q\in L^n(M,\C)$, we have the following mapping property, 
\[
L_{A,q}-(-\Delta_g): H^1(M^0)\to L^2(M). 
\] 
On the other hand, assuming that $A$ is merely of class $L^\infty(M,T^*M)$, we have in general the much weaker mapping property
\[
L_{A,q}-(-\Delta_g): H^1(M^0)\to H^{-1}(M^0), 
\] 
leading to some difficulties when trying to apply the techniques of boundary Carleman estimates. 

The plan of the paper is as follows. Section \ref{sec_Carleman} is devoted to Carleman estimates for magnetic Schr\"odinger operators with boundary terms and its consequences. Section \ref{sec_CGO} presents a construction of CGO solutions for the magnetic Schr\"odinger equation with $A \in (W^{1,n}\cap L^\infty)(M^0,T^*M^0)$ and $q \in L^n(M,\mathbb{C})$. The proof of Theorem \ref{thm1.1} is contained in Sections \ref{sec_recovery_magnetic} and \ref{sec_recovery_electric}. Theorem  \ref{thm1.2} is established in Section \ref{sec_Thm_1_2}. Section \ref{sec_Thm_1_3} is devoted to the proof of Theorem \ref{thm1.3}. The proof of Theorem \ref{thm1.4} will be omitted as it follows along the same lines as that of Theorem  \ref{thm1.2}. Appendix \ref{sec_app} contains some regularization estimates for Sobolev functions needed in the construction of CGO solutions.

\section{Carleman estimates}

\label{sec_Carleman}

Let $(M,g)$ be a compact smooth Riemaniann manifold of dimension $n\ge 3$ with smooth boundary. Assume that $(M,g)$ is CTA so that 
\[
(M,g)\subset\subset (\R\times M_0^0, g),\quad g=c(e\oplus g_0).
\]
Here $(M_0,g_0)$ is a compact $(n-1)$--dimensional manifold with boundary, $e$ is the Euclidian metric on $\R$, and $0<c\in C^\infty(M)$. We also write, here and in what follows, $X^0:=X\setminus \p X$ for the interior of a compact  smooth manifold with boundary $X$.  We define the semiclassical Sobolev norm of a function $v$ by $$\|v\|_{H^1_{scl}(M^0)}^2 = \|h\nabla v\|^2+\|v\|^2,$$ where we abbreviate the $L^2(M)$ norm by $\|\cdot\|$, and the corresponding $L^2$ scalar product by $(\cdot,\cdot)$. 

\bigskip
The following result is an extension of \cite{ferreira2006determining}*{Proposition 2.3} from the Euclidean setting to that of Riemannian manifolds, see also \cite{KnudsenSalo}*{Proposition 4.1}.  We refer to \cite{Kenig_Salo_2013}*{Proposition 4.1} for closely related Carleman estimates with boundary terms on CTA manifolds.

\begin{prop}
\label{prop2.1}
    Let $\varphi(x)=\pm x_1$ and let $\Tilde{\varphi} = \varphi + \frac{h}{2\epsilon}\varphi^2$. Then for all $0 < h \ll \epsilon \ll 1$ we have
    \begin{equation}
    \label{2.1}
        \|e^{\Tilde{\varphi} /h} \circ (-h^2\Delta_g) \circ e^{-\Tilde{\varphi} /h} u \|^2 \geq \frac{1}{C}\frac{h^2}{\epsilon}\|u\|_{H^1_{scl}(M^0)}^2-2h^3\int_{\partial M}(\partial_\nu \Tilde{\varphi})|\partial_\nu u|^2 dS_g, \hspace{5mm} C>0,
    \end{equation}
    for all $u\in H^2(M^0)\cap H_0^1(M^0)$. 
\end{prop}

\begin{proof}
 Defining 
 \[
  P_{\Tilde{\varphi}} := e^{\Tilde{\varphi} /h} \circ (-h^2\Delta_g) \circ e^{-\Tilde{\varphi} /h},
 \] 
 we get
 \[
   P_{\Tilde{\varphi}}   = -h^2 \Delta_g -|\nabla_g \Tilde{\varphi}|_g^2 +2 \langle \nabla_g\Tilde{\varphi},h \nabla_g \cdot\rangle_g + h (\Delta_g \Tilde{\varphi})= \Tilde{A}+i\Tilde{B}.
 \]
 Here $\Tilde{A}$ and $\Tilde{B}$ are formally self-adjoint operators on $L^2(M)$ given by
\begin{align*}
    \Tilde{A} &:= -h^2 \Delta_g - |\nabla_g \Tilde{\varphi}|_g^2\\
    \Tilde{B} &:= -2i \langle \nabla_g\Tilde{\varphi},h\nabla_g \cdot \rangle_g -ih (\Delta_g \Tilde{\varphi}).
\end{align*}
By density, it suffices to assume $u \in C^\infty(M^0)$ with $u|_{\partial M} =0$. Now note that
\begin{equation}
\label{2.2}
\begin{split}
    \|P_{\Tilde{\varphi}}u \|^2 &= (P_{\Tilde{\varphi}}u, P_{\Tilde{\varphi}}u)= ((\Tilde{A}+i\Tilde{B})u, (\Tilde{A}+i\Tilde{B})u)\\
    &= \|\Tilde{A}u\|^2 + \|\Tilde{B}u\|^2 - i(\Tilde{A}u,\Tilde{B}u) +i (\Tilde{B}u,\Tilde{A}u).
    \end{split}
\end{equation}
First we look at the term $(\Tilde{A}u,\Tilde{B}u)$. Since $\Tilde{B}$ is a first order differential operator and formally self-adjoint, we get by integration by parts and the fact that $u|_{\partial M} =0$, that 
\begin{equation}
\label{2.3}
    (\Tilde{A}u,\Tilde{B}u) = (\Tilde{B}\Tilde{A}u,u).
\end{equation}
Second we look at $(\Tilde{B}u,\Tilde{A}u)$. Since $\Tilde{A} = - h^2\Delta_g - |\nabla_g \Tilde{\varphi}|_g^2$, we use Green's formula to obtain
\begin{equation}
\label{2.4}
    (\Tilde{B}u,\Tilde{A}u) = (\Tilde{A}\Tilde{B}u,u) -h^2\int_{\partial M}(\partial_\nu \overline{u})(\Tilde{B}u)dS_g.
\end{equation}
Combining \eqref{2.3} and \eqref{2.4}, we can write \eqref{2.2} as
\begin{equation}
    \label{2.5}
    \|P_{\Tilde{\varphi}}u\|^2 = \|\Tilde{A}u\|^2+ \|\Tilde{B}u\|^2 +i([\Tilde{A},\Tilde{B}]u,u)-ih^2\int_{\partial M} (\partial_\nu \overline{u})(\Tilde{B}u)dS_g.
\end{equation}
Since $\Tilde{B}u|_{\partial M} = -2i \langle \nabla_g \Tilde{\varphi}, h\nabla_g u\rangle_g|_{\partial M} - \underbrace{ih(\Delta_g \Tilde{\varphi})u|_{\partial M}}_{=0 \text{ since } u|_{\partial M} =0}=-2ih\p_\nu \tilde \varphi\p_\nu u|_{\p M}$, we get 
\[
 -ih^2\int_{\partial M} (\partial_\nu \overline{u})(\Tilde{B}u)dS_g = -2h^3 \int_{\partial M}(\partial_\nu \Tilde{\varphi})|\partial_\nu u|^2dS_g.
\]
Thus \eqref{2.5} becomes
\begin{equation}
    \label{2.6}
    \|P_{\Tilde{\varphi}}u\|^2 = \|\Tilde{A}u\|^2+ \|\Tilde{B}u\|^2 +i([\Tilde{A},\Tilde{B}]u,u) -2h^3 \int_{\partial M}(\partial_\nu \Tilde{\varphi})|\partial_\nu u|^2dS_g.
\end{equation}
Recall from \cite{DKSaU_2009}*{page 143} that 
\[
i[\Tilde{A},\Tilde{B}] = 4\frac{h^2}{\epsilon}\bigg(1+\frac{h}{\epsilon}\varphi\bigg)^2+h\Tilde{B}\beta \Tilde{B}+h^2\Tilde{R},
\]
where $\beta = \frac{h}{\epsilon}(1+\frac{h}{\epsilon}\varphi)^{-2}$ and $\Tilde R$ is a semiclassical differential operator of order one, with coefficients uniformly bounded with respect to $h$ and $\epsilon$. Thus \eqref{2.6} becomes 
\begin{equation}
\label{2.9}
    \begin{split}
        \|P_{\Tilde{\varphi}}u\|^2 &= \|\Tilde{A}u\|^2 + \|\Tilde{B}u\|^2 + 4\frac{h^2}{\epsilon}\bigg(\bigg(1+\frac{h}{\epsilon}\varphi\bigg)^2u,u\bigg) +h (\Tilde{B}\beta \Tilde{B}u,u)\\
        &+h^2(\Tilde R u,u) - 2h^3 \int_{\partial M} (\partial_\nu \Tilde{\varphi})|\partial_\nu u|^2dS_g.
    \end{split}
\end{equation}
Here integrating by parts and using that  $u|_{\partial M} =0$, we get  
\[
(\Tilde{B} \beta \Tilde{B}u,u) = (\beta \Tilde{B}u,\Tilde{B}u).
\]
Noting also that $(1+\frac{h}{\epsilon}\varphi)^2 \geq \frac{1}{2}$ for $h$ sufficiently small, \eqref{2.9} implies that \begin{equation}
\label{2.10}
    \begin{split}
    \|P_{\Tilde{\varphi}}u\|^2 &\geq 2\frac{h^2}{\epsilon} \|u\|^2+\|\Tilde{A}u\|^2+\|\Tilde{B}u\|^2- \mathcal{O}(h)\|\Tilde{B}u\|^2\\
    &-\mathcal{O}(h^2)\|\Tilde{R}u\| \|u\|- 2h^3 \int_{\partial M} (\partial_\nu \Tilde{\varphi})|\partial_\nu u|^2dS_g.
    \end{split}
\end{equation}
Note that $\|\Tilde{B}u\|^2- \mathcal{O}(h)\|\Tilde{B}u\|^2\geq 0$ for $h$ sufficiently small. We also have
\[
\tilde R=\mathcal{O}(1): H^1_{scl}(M^0)\to L^2(M),
\]
and therefore, 
\[
\|\Tilde{R}u\| \|u\|\le \mathcal{O}(1) \|u\|_{H^1_{scl}(M^0)}^2.
\]
Thus we get from \eqref{2.10} that 
\begin{equation}
    \label{2.11}
    \|P_{\Tilde{\varphi}}u\|^2 \geq 2\frac{h^2}{\epsilon} \|u\|^2+\|\Tilde{A}u\|^2 -\mathcal{O}(h^2)\|u\|_{H^1_{scl}(M^0)}^2-2h^3\int_{\partial M} (\partial_\nu \Tilde{\varphi})|\partial_\nu u|^2dS_g.
\end{equation}
Now note that $(\Tilde{A}u,u) = (-h^2 \Delta_g u,u)-(|\nabla_g\Tilde{\varphi}|_g^2u,u)$. By integration by parts and the fact that $u|_{\partial M} =0$, we have $(-h^2 \Delta_g u,u) = \|h \nabla_g u\|^2$, and therefore, 
\[
\|h\nabla_g u\|^2 = (\Tilde{A}u,u) + (|\nabla_g\Tilde{\varphi}|_g^2u,u).
\]
 By the Cauchy-Schwarz inequality, we have $(\Tilde{A}u,u) \leq \|\tilde{A}u\|\|u\|$, and since $|\nabla_g \tilde{\varphi}|_g =\mathcal{O}(1)$, we get 
 \[
 \|h\nabla_g u\|^2  \leq \|\tilde{A}u\| \|u\| + \mathcal{O}(1)\|u\|^2 \leq C(\|\tilde{A}u\|^2+\|u\|^2),
 \]
 for some $C>1$.
Thus we obtain from  \eqref{2.11} that
\begin{equation}
    \label{2.12}
    \begin{split}
        \|P_{\Tilde{\varphi}}u\|^2 &\geq \bigg(1-\frac{h^2}{\epsilon}\bigg)\|\tilde{A}u\|^2 + \frac{h^2}{\epsilon} \bigg(\frac{1}{C}\|h\nabla_g u\|^2-\|u\|^2\bigg) +2\frac{h^2}{\epsilon}\|u\|^2\\
        &-\mathcal{O}(h^2)\|u\|_{H^1_{scl}(M^0)}^2-2h^3\int_{\partial M} (\partial_\nu \Tilde{\varphi})|\partial_\nu u|^2dS_g
    \end{split}
\end{equation}
Since $1-\frac{h^2}{\epsilon} \geq 0$ for $h$ sufficiently small, \eqref{2.12} implies that
\begin{equation}
    \label{2.13}
    \begin{split}
        \|P_{\Tilde{\varphi}}u\|^2 &\geq \frac{h^2}{\epsilon} \frac{1}{C}(\|h\nabla_g u\|^2+\|u\|^2) -2h^3\int_{\partial M} (\partial_\nu \Tilde{\varphi})|\partial_\nu u|^2dS_g\\
        &=  \frac{h^2}{\epsilon} \frac{1}{C} \|u\|_{H^1_{scl}(M^0)}^2-2h^3\int_{\partial M} (\partial_\nu \Tilde{\varphi})|\partial_\nu u|^2dS_g,
    \end{split}
\end{equation}
 for all  $0< h \ll \epsilon \ll 1$.
\end{proof}

\noindent  We shall now perturb the Carleman estimate of Proposition \ref{prop2.1} by lower order terms. The result below is an extension of \cite{KnudsenSalo}*{Proposition 4.1}  from the Euclidean setting to that of Riemannian manifolds. 
\begin{prop}
\label{prop2.2}
     Let $\varphi(x)=\pm x_1$ and let $\Tilde{\varphi} = \varphi + \frac{h}{2\epsilon}\varphi^2$. Let $A \in (W^{1,n} \cap L^\infty)(M^0,T^*M^0)$,  $q \in L^n(M,\mathbb{C})$. Then for $0 < h \ll \epsilon \ll 1$ we have
    \begin{equation}
    \label{2.15}
        \|e^{\Tilde{\varphi} /h} \circ (h^2 L_{A,q}) \circ e^{-\Tilde{\varphi} /h}u \|^2 \geq \frac{1}{C}\frac{h^2}{\epsilon}\|u\|_{H^1_{scl}(M^0)}^2-2h^3\int_{\partial M}(\partial_\nu \Tilde{\varphi})|\partial_\nu u|^2 dS_g, \hspace{5mm} C>0,
    \end{equation}
    for all $u\in H^2(M^0)\cap H_0^1(M^0)$. 
\end{prop}

\begin{proof}
   Let $u\in H^2(M^0)\cap H_0^1(M^0)$. Using \eqref{1.1}, we get 
       \begin{equation}
        \label{2.16}
        \begin{split}
            e^{\Tilde{\varphi} /h} \circ (h^2 L_{A,q}) \circ e^{-\Tilde{\varphi} /h} u&= e^{\Tilde{\varphi} /h} \circ (-h^2 \Delta_g) \circ e^{-\Tilde{\varphi} /h} u +ih^2(d^*A)u\\
            &-2ih^2 \langle A, e^{\tilde{\varphi}/h} \circ d \circ e^{-\tilde{\varphi}/h} u \rangle_g + (h^2 \langle A,A\rangle_g +h^2 q)u.
        \end{split}
    \end{equation}
    Now notice that by H\"older's inequality and the Sobolev embedding $H^1(M^0) \subset L^\frac{2n}{n-2}(M)$,   we have
\[
\|qu\|_{L^2(M)} \leq \|q\|_{L^n(M)}\|u\|_{L^\frac{2n}{n-2}(M)} \le \mathcal{O}(1)  \|q\|_{L^n}\|u\|_{H^1(M^0)}.
\] 
Since $h\|u\|_{H^1(M^0)} \leq \|u\|_{H^1_{scl}(M^0)}$, for $0<h\le 1$, we get
\begin{equation}
\label{2.17}
    h^2 \|qu\|_{L^2(M)} \leq \mathcal{O} (h) \|q\|_{L^n(M)}\|u\|_{H^1_{scl}(M^0)}.
\end{equation} 
Similarly, using that $d^*A \in L^n(M)$, we get for $0<h\le 1$,
\begin{equation}
    \label{2.18}
    h^2 \|(d^*A)u\|_{L^2(M)} \leq \mathcal{O}(h) \|d^*A\|_{L^n(M)}\|u\|_{H^1_{scl}(M^0)}.
\end{equation}
Furthermore, since $A\in L^\infty(M)$, we have
\begin{equation}
    \label{2.19}
    h^2 \|\langle A,A\rangle_g u\|_{L^2(M)} \leq h^2 \|\langle A,A\rangle_g \|_{L^\infty(M)} \|u\|_{L^2(M)} \leq \mathcal{O}(h^2)  \|u\|_{H^1_{scl}(M^0)}.
\end{equation}
Finally, we also have
\begin{equation}
    \label{2.20}
    \begin{split}
        h^2 \| \langle A, (e^{\tilde{\varphi}/h} \circ d \circ e^{-\tilde{\varphi}/h})u\rangle_g\|_{L^2(M)} &\leq h^2 \| \langle \underbrace{A}_{\in L^\infty(M)},du\rangle_g\|_{L^2(M)} +h \| \langle A, d\tilde{\varphi}\rangle_g u\|_{L^2(M)}\\
        &\leq \mathcal{O}(h) \|A\|_{L^\infty(M)}\|u\|_{H^1_{scl}(M^0)}.
    \end{split}
\end{equation}
Combining \eqref{2.1}, \eqref{2.17}, \eqref{2.18},\eqref{2.19}, \eqref{2.20}, and taking $\epsilon >0$ small enough, we obtain \eqref{2.15}.
\end{proof}
\noindent We shall need the following consequence of Proposition \ref{prop2.2}, obtained by taking $\epsilon>0$ sufficiently small but fixed, see \cite{ferreira2006determining}*{page 475}. We refer to  \cite{KnudsenSalo}*{Proposition 4.1} for the same result in the Euclidean case. Here
\begin{equation}
\label{eq_p_M-,+}
\p M_\pm^\varphi=\{x\in \p M: \pm \p_\nu \varphi(x)\ge 0\}.
\end{equation}
\begin{cor}
\label{cor_prop2.2}
     Let $\varphi(x)=\pm x_1$, $A \in (W^{1,n} \cap L^\infty)(M^0,T^*M^0)$, and  $q \in L^n(M,\mathbb{C})$. Then for $0 < h \ll  1$, we have
    \begin{equation}
    \label{2.15_cor}
\begin{aligned}
        -h(\p_\nu \varphi\, e^{\varphi /h}\p_\nu u,&e^{\varphi /h} \p_\nu u)_{L^2(\p M_-^\varphi)}+ \|e^{\varphi /h} u\|_{H^1_{scl}(M^0)}^2\\
       & \le \mathcal{O}(h^{-2})
        \|e^{\varphi /h} (h^2 L_{A,q}) u \|^2 +\mathcal{O}(h)(\p_\nu \varphi\, e^{\varphi /h}\p_\nu u, e^{\varphi /h}\p_\nu u)_{L^2(\p M_+^\varphi)}, 
\end{aligned}
    \end{equation}
    for all $u\in H^2(M^0)\cap H_0^1(M^0)$. 
\end{cor}

\noindent In order to construct CGO solutions, we shall also need interior Carleman estimates. Our starting point is the following Carleman estimate for $-h^2\Delta_g$ with a gain of two derivatives, obtained in \cite{2018}*{Proposition 2.2}. Here and in what follows, we set
\[
\|u\|_{H_{scl}^{-1}(M^0)}=\sup_{0\ne \psi\in C^\infty_0(M^0)}\frac{|\langle u, \psi\rangle_{M_0}|}{\|\psi\|_{H^1_{scl}(M_0)}}.
\]
\begin{prop}
\label{cor2.1}
 Let $\varphi(x)=\pm x_1$ and let $\Tilde{\varphi} = \varphi + \frac{h}{2\epsilon}\varphi^2$.  Then for $0 < h \ll \epsilon \ll 1$, we have
    \begin{equation}
    \label{2.14}
        \frac{h}{\sqrt{\epsilon}} \|u\|_{H^1_{scl}(M^0)} \leq C \| e^{\tilde{\varphi} / h}\circ (-h^2 \Delta_g) \circ e^{-\tilde{\varphi} / h}u\|_{H_{scl}^{-1}(M^0)}, \hspace{5mm} C>0,
    \end{equation}
    for all $u\in C_0^\infty(M^0)$.
\end{prop}
\noindent  We shall now perturb the Carleman estimate of Proposition \ref{cor2.1} by lower order terms. To that end, using \eqref{2.16},  \eqref{2.14}, \eqref{2.17}, \eqref{2.18},\eqref{2.19}, \eqref{2.20}, and taking $\epsilon>0$ sufficiently small but fixed, we get the following result. 
\begin{cor}
\label{cor_cor2.1}
     Let $\varphi(x)=\pm x_1$, $A \in (W^{1,n} \cap L^\infty)(M^0,T^*M^0)$, and  $q \in L^n(M,\mathbb{C})$. Then for $0 < h \ll  1$, we have
    \begin{equation}
    \label{2.14_cor}
        h \|u\|_{H^1_{scl}(M^0)} \leq C \| e^{\varphi / h}\circ (h^2 L_{A,q}) \circ e^{-\varphi / h}u\|_{H_{scl}^{-1}(M^0)}, \hspace{5mm} C>0,
    \end{equation}
    for all $u\in C_0^\infty(M^0)$.
\end{cor}

\noindent We shall also need the following solvability result. 
\begin{theorem}
\label{thm_solvability}
     Let $\varphi(x)=\pm x_1$, $A \in (W^{1,n} \cap L^\infty)(M^0,T^*M^0)$, and  $q \in L^n(M,\mathbb{C})$. If $h>0$ is small enough, then for any $v\in H^{-1}(M^0)$, there is a solution $u\in H^1(M^0)$ of the equation 
\[     
e^{\varphi/h}(h^2 L_{A,q}) e^{-\varphi/h} u = v \text{ in } M^0,
\]
which satisfies $$\|u\|_{H^1_{scl}(M^0)} \leq \frac{C}{h}\|v\|_{H^{-1}_{scl}(M^0)}.$$
\end{theorem}

\begin{proof}
 The proof follows from the standard argument of the Hahn-Banach theorem and  Carleman estimates of Corollary  \ref{cor_cor2.1}, see \cite{UhlmannKatya2014}*{Proposition 2.3}.
\end{proof}

\section{Construction of CGO solutions}

\label{sec_CGO}

Let $(M,g)$ be an admissible manifold so that $(M,g)\subset\subset (\R\times M_0^0,c(e\oplus g_0))$, where $(M_0,g_0)$ is simple. Replacing $M_0$ by slightly large simple manifold if needed, we may assume that for some simple manifold $(D,g_0)$, we have
\begin{equation}
\label{eq_3_emb}
(M,g) \subset \subset (\mathbb{R}\times D^0, c(e \oplus g_0)) \subset \subset (\mathbb{R} \times M_0^0, c(e \oplus g_0)). \end{equation}

\noindent Let $A \in (W^{1,n}\cap L^\infty)(M^0,T^*M^0)$, and let us extend $A$ to all of $\R\times M_0^0$ so that the extension, denoted by the same letter, satisfies $A \in (W^{1,n}\cap L^\infty)(\R\times M_0^0,T^*(\R\times M_0^0))$ and is compactly supported,  see \cite{Brezis_book}*{Theorem 9.7}. 
Let  $q\in L^n(M,\mathbb{C})$, and let us take the zero extension of $q$ to all of $\R\times M_0^0$.  We denote the extension by the same letter.

\noindent Using a partition of unity argument combined with Propositions \ref{propA.1},  \ref{propA.3}, we get the following result to be used when constructing the CGO solutions below. 
\begin{prop}
\label{prop_A-reg}
There exists a family $A_\tau\in C^\infty_0(\R\times M_0^0,T^*(\R\times M_0^0) )$, $\tau>0$, such that 
\begin{equation}
\label{prop_A-reg_1}
\|A-A_\tau\|_{L^n}=o(\tau),
\end{equation}
\begin{equation}
\label{prop_A-reg_2}
\|A_\tau\|_{L^n}=\mathcal{O}(1), \quad \|\nabla_g A_\tau^\sharp\|_{L^n}=\mathcal{O}(1), \quad \|\nabla_g^2A_\tau^\sharp\|_{L^n}=o(\tau^{-1}),
\end{equation}
and 
\begin{equation}
\label{prop_A-reg_3}
\|\nabla_g^k A_\tau^\sharp \|_{L^\infty}=\mathcal{O}(\tau^{-k}), \quad k=0,1,2,\dots,
\end{equation}
as $\tau\to 0$. Here $A^\sharp_\tau=\sum_{j,k=1}^n g^{jk}(A_\tau)_k\p_{x_j}$ is the vector field, corresponding to the one form $A_\tau$. 
\end{prop}

We have global coordinates $x = (x_1,x') \in \mathbb{R}\times M_0$ in which the metric $g$ has the form 
 \begin{equation}
    \label{3.1}
     g(x) = c(x) \begin{pmatrix}1 & 0 \\ 0 & g_0(x') \end{pmatrix},
 \end{equation}
 where $c>0$ and $g_0$ is a simple metric on $M_0$. 
Let $\varphi(x) = x_1$ and observe that this is a limiting Carleman weight in the sense of \cite{DKSaU_2009}.

We shall next construct CGO solutions to the equation 
 \begin{equation}
    \label{3.2}
 L_{A,q}u=0  \text{ in } \mathcal{D}'(M^0),
 \end{equation}
 of the form 
  \begin{equation}
    \label{3.3}
 u=e^\frac{-\rho}{h}(a+r),
 \end{equation}
  where $\rho = \varphi + i \psi$ is a complex phase, $\psi \in C^\infty(M,\mathbb{R})$, $a \in C^\infty(M,\mathbb{C})$ is an amplitude and $r$ is a remainder term. To that end, we  have 
 $$e^\frac{\rho}{h} \circ (-h^2 \Delta_g) \circ e^{-\frac{\rho}{h}} = -h^2\Delta_g + h\Delta_g \rho +2h \nabla_g \rho - |\nabla_g \rho|_g^2$$
 where $\nabla_g \rho$ is a complex vector field and $|\nabla_g \rho|_g^2 = \langle \nabla_g \rho , \nabla_g \rho \rangle_g$ is computed using the bilinear extension of the Riemaniann scalar product to the complexified tangent bundle. Also we have 
 $$e^\frac{\rho}{h}ih^2d^*(Ae^{-\frac{\rho}{h}}a) = ih^2d^*(Aa) +ih \langle A, d\rho\rangle_g a,$$
 and
 $$-ih^2e^\frac{\rho}{h} \langle A, d(e^{-\frac{\rho}{h}}a)\rangle_g = -ih^2 \langle A, da\rangle_g + ih \langle A, d\rho \rangle_g a.$$
 So conjugating the magnetic Schr\"odinger operator and writing $A=(A-A_\tau)+A_\tau$, where $A_\tau\in C^\infty_0(\R\times M_0^0,T^*(\R\times M_0^0) )$ is given in Proposition  \ref{prop_A-reg}, we get
 \begin{align*}
 e^\frac{\rho}{h}h^2 L_{A,q}(e^{-\frac{\rho}{h}} a) &= -h^2\Delta_g a + h(\Delta_g \rho)a + 2h \nabla_g \rho(a) - |\nabla_g \rho|_g^2 a- ih^2\langle A, da\rangle_g  \\
 &+2ih \langle (A-A_\tau), d\rho \rangle_g a +2ih \langle A_\tau, d\rho \rangle_g a +ih^2d^*(Aa)+h^2(\langle A,A\rangle_g +q)a.
  \end{align*}
  For \eqref{3.3} to be a solution to \eqref{3.2} we require $\rho$ to satisfy the eikonal equation 
  \begin{equation}
    \label{3.4}
      |\nabla_g \rho|_g^2 =0,
  \end{equation}
  and the amplitude $a$ to satisfy the regularized transport equation 
  \begin{equation}
    \label{3.5}
      (2i\langle A_\tau, d\rho \rangle_g + 2\nabla_g \rho)a + (\Delta_g \rho)a =0. 
  \end{equation}
  The remainder $r$ is determined by 
  \begin{equation}
    \label{3.6}
  \begin{split}
      e^\frac{\rho}{h}h^2 L_{A,q}(e^{-\frac{\rho}{h}}r) &= -(-h^2\Delta_g a-ih^2\langle A,da\rangle_g +2ih\langle A-A_\tau,d\rho\rangle_g a+ih^2d^*(Aa)\\
      &+h^2(\langle A,A\rangle_g+q)a).
      \end{split}
  \end{equation}
First, following \cite{DKSaU_2009}*{Section 5} and \cite{2018}*{Section 3}, we let $\omega\in D$ be a point such that $(x_1,\omega)\notin M$ for all $x_1\in \R$. We have the global coordinates on $M$ given by $x=(x_1,r,\theta)$, where  $(r,\theta)$ are the polar normal coordinates on $(D,g_0)$ with center $\omega$, i.e. $x' = \exp _\omega^D(r\theta)$ where $r>0$ and $\theta \in \mathbb{S}^{n-2}$. Here $\exp_\omega^D$ is the exponential map which takes its maximal domain in $T_\omega D$ diffeomorphically onto $D$ since $D$ is simple.  Following \cite{DKSaU_2009}*{Section 5} and \cite{2018}*{Section 3}, we  take 
\[
\rho = x_1 +i r.
\] 

As in \cite{2018}*{Section 3},  the transport equation \eqref{3.5} in the coordinates $x=(x_1,r,\theta)$  becomes   \begin{equation}
    \label{3.11}
      4\overline{\partial}a + \bigg(\overline{\partial} \log\bigg(\frac{|g|}{c^2}\bigg)\bigg)a+2i((A_\tau)_1+i(A_\tau)_r)a = 0,
  \end{equation}
  where $\bar\p =\frac{1}{2}(\p_{x_1}+i\p_r)$. Following \cite{2018}*{Section 3}, we choose a solution $a$ of the form 
  \begin{equation}
  \label{3.12}
      a = |g|^{-\frac{1}{4}}c^\frac{1}{2}e^{i\Phi_\tau}a_0(x_1,r)b(\theta),
  \end{equation}
  where the function $a_0$ is non-vanishing holomorphic so that $\overline{\partial}a_0 =0$,  the function $b(\theta)$ is smooth, and $\Phi_\tau$ solves a $\overline{\partial}$-equation, 
  \begin{equation}
  \label{3.13}
      \overline{\partial} \Phi_\tau = -\frac{1}{2}((A_\tau)_1+i(A_\tau)_r).
  \end{equation}
  Note that the right hand side of \eqref{3.13} belongs to $C^\infty_0(\R^2)$.  The equation \eqref{3.13} is given in the global coordinates $(x_1,r)$ so using the fundamental solution $\frac{1}{\pi(x_1+ir)}\in L^1_{loc}(\mathbb{R}^2)$ of the $\overline{\partial}$ operator we can take 
  \begin{equation}
    \label{3.14}
      \Phi_\tau(x_1,r,\theta) = -\frac{1}{2}\frac{1}{\pi(x_1+ir)}*((A_\tau)_1+i(A_\tau)_r).
  \end{equation}
  Here $*$ denotes the convolution in the variables $(x_1,r)$. 
  
 We shall now derive some estimates for the function $\Phi_\tau$ and its first two derivatives.  First using \eqref{prop_A-reg_3}, we get  
   \begin{equation}
    \label{3.15}
      \|\Phi_\tau\|_{L^\infty(M)} = \mathcal{O}(1), \hspace{5mm} \|\nabla_g \Phi_\tau\|_{L^\infty(M)} = \mathcal{O}(\tau^{-1}), \hspace{5mm} \|\Delta_g \Phi_\tau\|_{L^\infty(M)} = \mathcal{O}(\tau^{-2}), 
  \end{equation}
  as $\tau \to 0$. Next using \eqref{prop_A-reg_2} and Young's inequality, we obtain that 
 \begin{equation}
    \label{3.17}
      \|\Phi_\tau\|_{L^n(M)}=\mathcal{O}(1), \quad   \|\nabla_g\Phi_\tau\|_{L^n(M)}=\mathcal{O}(1), \quad  \|\Delta_g\Phi_\tau\|_{L^n(M)}=o(\tau^{-1}),
  \end{equation}
as $\tau\to 0$. For future reference, let us introduce 
 \[
 \Phi(x_1,r,\theta) = -\frac{1}{2} \frac{1}{\pi(x_1+ir)}* (A_1+iA_r) \in L^\infty(M).
 \]
 Using \eqref{prop_A-reg_1}, the fact that $A$ and $A_\tau$ have compact support in $\R\times M_0^0$,  and Young's inequality, we see that 
  \begin{equation}
    \label{3.16}
      \|\Phi_\tau - \Phi\|_{L^n(M)} = o(\tau),
  \end{equation}
  as $\tau \to 0$.

We shall next establish some estimates for the amplitude $a$ given by \eqref{3.12} and its first two derivatives in $L^\infty$ and $L^2$, which will be sufficient for our purposes. First \eqref{3.15} gives that 
 \begin{equation}
    \label{3.15_a}
      \|a\|_{L^\infty(M)} = \mathcal{O}(1), \hspace{5mm} \|\nabla_g a\|_{L^\infty(M)} = \mathcal{O}(\tau^{-1}), \hspace{5mm} \|\Delta_g a\|_{L^\infty(M)} = \mathcal{O}(\tau^{-2}), 
  \end{equation}
  as $\tau \to 0$. Furthermore, using \eqref{3.17} and \eqref{3.12}, and that fact that $L^n(M)\subset L^2(M)$,  we get 
 \begin{equation}
    \label{3.17_a}
      \|a\|_{L^2(M)}=\mathcal{O}(1), \quad   \|\nabla_g a\|_{L^2(M)}=\mathcal{O}(1),
  \end{equation}
as $\tau\to 0$. Next we shall bound $\Delta_g a$ in $L^2(M)$, and in view of  \eqref{3.12}, it suffices to estimate $\Delta_g e^{i\Phi_\tau}$ in $L^2(M)$.  In doing so, we note that 
\begin{equation}
    \label{3.17_lap}
\Delta_g e^{i\Phi_\tau} = e^{i\Phi_\tau}(i\Delta_g \Phi_\tau - \langle \nabla_g \Phi_\tau, \nabla_g \Phi_\tau \rangle_g).
\end{equation} 
Estimating the first term in the right hand side of \eqref{3.17_lap},  using \eqref{3.17} and \eqref{3.15},  we get
\begin{equation}
    \label{3.17_lap_1}
  \|e^{i\Phi_\tau} \Delta_g \Phi_\tau\|_{L^2} \leq \|e^{i\Phi_\tau}\|_{L^\infty}\|\Delta_g \Phi_\tau\|_{L^2} \leq C \underbrace{\|e^{i\Phi_\tau}\|_{L^\infty}}_{=\mathcal{O}(1)} \underbrace{\|\Delta_g \Phi_\tau\|_{L^n}}_{=o(\tau^{-1})}= o(\tau^{-1}).
\end{equation}  
We shall now estimate the second term in \eqref{3.17_lap}. First we have 
\begin{equation}
    \label{3.17_lap_2}
\|e^{i\Phi_\tau} \langle \nabla_g \Phi_\tau , \nabla_g \Phi_\tau\ \rangle_g \|_{L^2} \leq \underbrace{\|e^{i\Phi_\tau}\|_{L^\infty}}_{=\mathcal{O}(1)} \underbrace{\|\langle \nabla_g \Phi_\tau , \nabla_g \Phi_\tau \rangle_g \|_{L^2}}_{\leq \|\nabla_g \Phi_\tau\|_{L^4}^2 \text{ by Cauchy-Schwarz}} \leq C \|\nabla_g \Phi_\tau\|_{L^4}^2.
\end{equation}

Note that when  $n\geq 4$, by H\"older's inequality,  we have 
\begin{equation}
    \label{3.17_lap_3}
\|\nabla_g \Phi_\tau\|_{L^4} \leq C\|\nabla_g \Phi_\tau\|_{L^n} = \mathcal{O}(1).
\end{equation}
   When  $n=3$, using H\"older's inequality and \eqref{3.17}, \eqref{3.15},  we get 
\begin{equation}
    \label{3.17_lap_4}
 \|\nabla_g \Phi_\tau\|_{L^4}^2 \leq \underbrace{\|\nabla_g \Phi_\tau\|_{L^\infty}^{1/2}}_{= \mathcal{O}(\tau^{-1/2})} \underbrace{\|\nabla_g \Phi_\tau\|_{L^3}^{3/2}}_{=\mathcal{O}(1)}= \mathcal{O}(\tau^{-1/2})  \leq o(\tau^{-1}).
 \end{equation}
 Combining \eqref{3.17_lap},  \eqref{3.17_lap_1},  \eqref{3.17_lap_2},   \eqref{3.17_lap_3}, and  \eqref{3.17_lap_4}, we obtain that    
 \begin{equation}
            \label{3.22}
       \|\Delta_g a\|_{L^2} \le \mathcal{O}(1) \|\Delta_g e^{i\Phi_\tau}\|_{L^2} +\mathcal{O}(1)= o(\tau^{-1}).
 \end{equation}

  Now we will solve \eqref{3.6} for the remainder term $r$. First note that the right hand side of \eqref{3.6} is given by 
  \begin{equation}
    \label{3.23}
      \begin{split}
          v &:= -(-h^2\Delta_g a-ih^2\langle A,da\rangle_g +2ih\langle A-A_\tau,d\rho\rangle_g a+ih^2d^*(Aa)\\
            &+h^2(\langle A,A\rangle_g+q)a).
      \end{split}
  \end{equation}
  Using that 
 \[
 d^*(Aa)=(d^*A)a-\langle A,da\rangle_g ,
 \] 
 we get  
   \begin{equation}
    \label{3.23_new}
      \begin{split}
          v &= -(-h^2\Delta_g a-2ih^2\langle A,da\rangle_g +2ih\langle A-A_\tau,d\rho\rangle_g a+ih^2(d^*A)a\\
            &+h^2(\langle A,A\rangle_g+q)a).
      \end{split}
  \end{equation}
  
  We first estimate $\|v\|_{L^2(M)}$. First we have by \eqref{3.22},
  \begin{equation}
    \label{3.24}
      \|-h^2\Delta_g a\|_{L^2}  = h^2o(\tau^{-1}).
  \end{equation}
  Secondly, by \eqref{3.17_a},  we have
  \begin{equation}
    \label{3.25}
      \|h^2\langle A, da\rangle_g\|_{L^2} =\mathcal{O}( h^2) \|A\|_{L^\infty} \|da\|_{L^2} = \mathcal{O}(h^2).
  \end{equation}
  Then  using Cauchy-Schwarz inequality, \eqref{prop_A-reg_1}, and \eqref{3.15_a}, we have
  \begin{equation}
    \label{3.26}
      \|h\langle A-A_\tau, d\rho\rangle_g a\|_{L^2} \le \mathcal{O}(h)\|a\|_{L^\infty} \|A-A_\tau\|_{L^2} = o(\tau h).
  \end{equation}
 Using that $d^*A\in L^n(M)$ and that $L^n(M)\subset L^2(M)$, we get   
   \begin{equation}
    \label{3.28}
      \|h^2 (d^*A+q)a\|_{L^2} \le  \mathcal{O}(h^2)\|a\|_{L^\infty}\|d^*A+q\|_{L^n} =\mathcal{O}(h^2).
  \end{equation}
 Finally, we have  
 \begin{equation}
    \label{3.28_aa}
      \|h^2  \langle A,A\rangle_g a\|_{L^2}=\mathcal{O}(h^2).
  \end{equation}  
  Putting together \eqref{3.23_new},  \eqref{3.24}, \eqref{3.25}, \eqref{3.26}, \eqref{3.28}, \eqref{3.28_aa}, we get 
  \begin{equation}
    \label{3.29}
      \|v\|_{L^2(M)} =o(h^2\tau^{-1}+\tau h) +\mathcal{O}(h^2) = o(h(\tau +h \tau^{-1})).
  \end{equation}
  Note $\tau \to \tau + h \tau^{-1}$ has minimum at $\tau = \sqrt{h}$ so using this we get $\|v\|_{L^2} =o(h^{3/2})$. So by Theorem \ref{thm_solvability}, we obtain the existence of $r\in H^1(M)$ with $\|r\|_{H^1_{scl}(M^0)} = o(h^{1/2})$.

Let us summarize the discuss of this section in the following proposition. 
\begin{prop}
\label{prop3.1}
Assume that $(M,g)$ satisfies \eqref{eq_3_emb} and \eqref{3.1}. Let $A \in (W^{1,n}\cap L^\infty)(M^0,T^*M^0)$,  $q\in L^n(M,\mathbb{C})$. Let $\omega\in D$ be a point such that $(x_1,\omega)\notin M$ for all $x_1\in \R$, and let $(r,\theta)$ be the polar normal coordinates on $(D,g_0)$ with center $\omega$. Then for all $h>0$ small enough, there exists a solution $u \in H^1(M^0)$ to 
 \[
 L_{A,q}u=0  \text{ in } \mathcal{D}'(M^0),
 \]
 of the form 
  \[
 u=e^\frac{-\rho}{h}(a+r),
 \]
  where $\rho = x_1 + i r$,  
 \[
  a = |g|^{-\frac{1}{4}}c^\frac{1}{2}e^{i\Phi_h}a_0(x_1,r)b(\theta),
  \]
  $a_0$ is a non-vanishing holomorphic function so that $(\p_{x_1}+i\p_r)a_0 =0$,  $b(\theta)$ is smooth. Here $\Phi_h\in C^\infty(M)$ such that 
 \begin{align*}
      \|a\|_{L^\infty(M)} = \mathcal{O}(1), \hspace{5mm} \|\nabla_g a\|_{L^\infty(M)} = \mathcal{O}(h^{-1/2}), \hspace{5mm} \|\Delta_g a\|_{L^\infty(M)} = \mathcal{O}(h^{-1}), \\
 \|a\|_{L^2(M)}=\mathcal{O}(1), \quad   \|\nabla_g a\|_{L^2(M)}=\mathcal{O}(1),\quad  \|\Delta_g a\|_{L^2} = o(h^{-1/2}).
  \end{align*}
as $h\to 0$, and 
\[
 \|\Phi_h - \Phi\|_{L^n(M)} = o(h^{1/2}),
\]
as $h\to 0$, 
where 
\[
 \Phi(x_1,r,\theta) = -\frac{1}{2} \frac{1}{\pi(x_1+ir)}* (A_{x_1}+iA_r) 
 \]
 with $A=A_{x_1}dx_1+A_rdr+A_\theta d\theta$. Furthermore, the remainder $r$ is such that $\|r\|_{H^1_{scl}(M^0)} = o(h^{1/2})$, as $h\to 0$.
\end{prop}

\section{Recovering the magnetic field}

\label{sec_recovery_magnetic}

We shall start this section by making several general observations valid on a general smooth compact Riemannian manifold with boundary.

\begin{prop}
\label{prop_trace_magnetic}
Let $(M,g)$ be a smooth compact Riemannian manifold of dimension $n\ge 3$ with boundary, and let $\nu$ be the unit outer normal to $\p M$. 
Let $A\in W^{1,n}(M^0, T^*M^0)$ and let $u\in H^1(M^0)$. Then we have $(\langle A,\nu\rangle_g u)|_{\p M}\in H^{-\frac{1}{2}}(\p M)$. 
\end{prop}
\begin{proof}
By the Sobolev trace theorem, we have $u|_{\p M}\in  H^{\frac{1}{2}}(\p M)$. Furthermore, 
\[
\langle A,\nu\rangle_g|_{\p M}\in  W^{1-\frac{1}{n},n}(\p M)\subset L^n(\p M),
\]
see \cite{Brezis_book}*{page 315}.  Let $v\in H^{\frac{1}{2}}(\p M)$.  By Sobolev's embedding, we have 
\begin{equation}
\label{eq_magn_sob}
H^{\frac{1}{2}}(\p M) \subset  L^{\frac{2(n-1)}{n-2}}(\p M)
\end{equation}
see \cite{Sogge_book}*{Theorem 0.3.8, page 28}. Using  H\"older's inequality and \eqref{eq_magn_sob}, we get 
\begin{equation}
\label{eq_4_trace_A_bound}
\begin{aligned}
\bigg|\int_{\p M} \langle A,\nu\rangle_g uvdS_g\bigg|&\le \|\langle A,\nu\rangle_g\|_{L^{n-1}(\p M)}\|u\|_{L^{\frac{2(n-1)}{n-2}}(\p M)}\|v\|_{L^{\frac{2(n-1)}{n-2}}(\p M)}\\
&\le  C\|\langle A,\nu\rangle_g\|_{L^{n}(\p M)}\|u\|_{H^{\frac{1}{2}}(\p M)}\|v\|_{H^{\frac{1}{2}}(\p M)},
\end{aligned}
\end{equation}
showing the result. 
\end{proof}

\begin{prop}
\label{prop_trace_magnetic_2}
Let $(M,g)$ be a smooth compact Riemannian manifold of dimension $n\ge 3$ with boundary, and let $\nu$ be the unit outer normal to $\p M$. 
Let $A\in W^{1,n}(M^0, T^*M^0)$. Then $\langle A,\nu\rangle_g|_{\p M}\in L^p(\p M)$ for all $1\le p< \infty$. 
\end{prop}

\begin{proof}
We have $A\in W^{1,q}(M^0, T^*M^0)$ for all $1\le q\le n$. Thus, by the trace theorem, see \cite{Brezis_book}*{page 315}, and Sobolev's embedding, see \cite{Sogge_book}*{Theorem 0.3.8, page 28}, we get 
\[
\langle A,\nu\rangle_g|_{\p M}\in W^{1-\frac{1}{q},q}(\p M)\subset L^\frac{q(n-1)}{n-q}(\p M), \quad 1\le q< n,
\]
 which shows the result. 
\end{proof}

We shall need the following Green formula for the magnetic Schr\"odinger operator $L_{A,q}$ with $A\in W^{1,n}(M^0, T^*M^0)$ and $q\in L^n(M,\mathbb{C})$. 

\begin{prop}
Let $(M,g)$ be a smooth compact Riemannian manifold of dimension $n\ge 3$ with boundary, and let $\nu$ be the unit outer normal to $\p M$. Let  $A\in W^{1,n}(M^0, T^*M^0)$ and $q\in L^n(M,\mathbb{C})$. We have the magnetic Green formula, 
\begin{equation}
\label{eq_4_Green_1}
\begin{aligned}
(L_{A,q} u, &v)_{L^2(M)}- (u, L_{\overline{A},\overline{q}} v)_{L^2(M)}\\
&=-\langle \p_\nu u+ i \langle A,\nu\rangle_gu, \overline{v}\rangle_{H^{-\frac{1}{2}}(\p M)\times H^{\frac{1}{2}}(\p M)}+ \langle u, \overline{\p_\nu v+ i \langle \overline{A},\nu\rangle_gv}\rangle_{H^{\frac{1}{2}}(\p M)\times H^{-\frac{1}{2}}(\p M)},
\end{aligned}
\end{equation}
for all $u, v\in H^1(M^0)$ such that $\Delta_g u, \Delta_g v\in L^2(M)$. Here $\langle\cdot,\cdot\rangle_{H^{-\frac{1}{2}}(\p M)\times H^{\frac{1}{2}}(\p M)}$ is the distributional duality between $H^{-\frac{1}{2}}(\p M)$ and $H^{\frac{1}{2}}(\p M)$. 
\end{prop}

\begin{proof}
Let $u, v\in H^1(M^0)$ be such that $\Delta_g u, \Delta_g v\in L^2(M)$. 
We have 
\begin{equation}
\label{eq_4_Green_2}
(-\Delta_g u, v)_{L^2(M)}- (u, (-\Delta_g v))_{L^2(M)}= \langle u, \overline{\p_\nu v}\rangle_{H^{\frac{1}{2}}(\p M)\times H^{-\frac{1}{2}}(\p M)}-\langle \p_\nu u, \overline{v}\rangle_{H^{-\frac{1}{2}}(\p M)\times H^{\frac{1}{2}}(\p M)},
\end{equation}
see \cite{Bukhgeim_Uhlmann_2002}, \cite{ferreira2006determining}. We also have
\begin{equation}
\label{eq_4_Green_2_1}
  ( i(d^*A)u +(\langle A,A\rangle_g +q)u, v)_{L^2(M)}=(u, -i(d^*\overline{A})v+  (\langle \overline{A},\overline{A}\rangle_g +\overline{q}) v)_{L^2(M)}.
\end{equation}  
 Assuming first that $A\in C^\infty(M, T^*M)$ and computing in local coordinates, we get 
\begin{equation}
\label{eq_4_Green_3}
\begin{aligned}
(\langle A, d u\rangle_g, v)_{L^2(M)}&=(u, d^*(\overline{A}v))_{L^2(M)}+\int_{\p M} \langle A,\nu\rangle_g u\overline{v}dS_g\\
&=(u, (d^*\overline{A})v)_{L^2(M)}- (u, \langle \overline{A}, dv\rangle_g)_{L^2(M)}+\int_{\p M} \langle A,\nu\rangle_g u\overline{v}dS_g.
\end{aligned}
\end{equation}
Here we have used  \eqref{eq_1.1-useful_mag}.  Let us show that \eqref{eq_4_Green_3} extends to $A\in W^{1,n}(M^0, T^*M^0)$. In doing so, by H\"older's inequality and Sobolev's embedding, we get 
\begin{equation}
\label{eq_4_Green_4}
\begin{aligned}
&|(\langle A, d u\rangle_g, v)_{L^2(M)}|\le \|A\|_{L^n(M)}\|du\|_{L^2(M)}\|v\|_{L^{\frac{2n}{n-2}}(M)}\le C\|A\|_{L^n(M)}\|u\|_{H^1(M^0)}\|v\|_{H^1(M^0)},\\
&|(u, \langle \overline{A}, dv\rangle_g)_{L^2(M)}|\le C\|A\|_{L^n(M)}\|u\|_{H^1(M^0)}\|v\|_{H^1(M^0)},\\
&|(u, (d^*\overline{A})v)_{L^2(M)}|\le \|d^*A\|_{L^n(M)}\|u\|_{L^2(M)}\|v\|_{L^{\frac{2n}{n-2}}(M)} \le C\|d^*A\|_{L^n(M)}\|u\|_{H^1(M^0)}\|v\|_{H^1(M^0)}.
\end{aligned}
\end{equation}
Using the density of $C^\infty(M, T^*M)$ in $W^{1,n}(M^0, T^*M^0)$, together with the bounds \eqref{eq_4_Green_4} and \eqref{eq_4_trace_A_bound}, we obtain the claim. 

\noindent 
It follows from \eqref{eq_4_Green_3} that 
\begin{equation}
\label{eq_4_Green_5}
\begin{aligned}
(-2i \langle A, d u\rangle_g, v)_{L^2(M)}=(u, 2i (d^*\overline{A})v)_{L^2(M)}- (u, 2i\langle \overline{A}, dv\rangle_g)_{L^2(M)}\\
-\int_{\p M} i \langle A,\nu\rangle_g u\overline{v}dS_g+ \int_{\p M} u\overline{i \langle \overline{A},\nu\rangle_g  v}dS_g.
\end{aligned}
\end{equation}
Combining \eqref{eq_4_Green_2}, \eqref{eq_4_Green_2_1}, and \eqref{eq_4_Green_5}, in view of \eqref{1.1}, we obtain 
\eqref{eq_4_Green_1}. 
\end{proof}

In what follows we shall let $(M,g)$ be an admissible simply connected manifold of dimension $n\ge 3$ with connected boundary.  We shall need the following integral identity. 
\begin{prop} 
\label{prop_integral_identity}
Let  $A_j\in (W^{1,n}\cap L^\infty)(M^0, T^*M^0)$ and $q_j\in L^n(M,\mathbb{C})$, $j=1,2$. Assume that $C_{A_1,q_1}^\Gamma=C_{A_2,q_2}^\Gamma$. Then we have 
\begin{equation}
\label{eq_10_6_prop}
\begin{aligned}
\int_M i \langle A_1-A_2, u_1d\overline{u_2}-\overline{u_2}du_1\rangle_g dV_g+ \int_M (\langle A_1, A_1\rangle_g -\langle A_2,A_2\rangle_g+q_1-q_2)u_1\overline{u_2}dV_g\\
=-\int_{\p M\setminus \Gamma}\p_\nu (w_2-u_1)\overline{u_2}dS_g+ i \int_{\p M\setminus\Gamma} \langle  A_1-A_2,\nu\rangle_g u_1\overline{u_2}dS_g,
\end{aligned}
\end{equation}
for $u_1,u_2\in H^1(M^0)$ satisfying 
\begin{equation}
\label{eq_10_1}
L_{A_1,q_1} u_1=0, \quad L_{\overline{A_2}, \overline{q_2}}u_2=0, \quad \text{in}\quad \mathcal{D}'(M^0),
\end{equation}
and $w_2\in H^1(M^0)$ satisfying \begin{equation}
\label{eq_10_2}
L_{A_2,q_2} w_2=0 \quad \text{in}\quad \mathcal{D}'(M^0),
\end{equation}
such that
\begin{equation}
\label{eq_10_3}
w_2|_{\p M}=u_1|_{\p M}, \quad (\p_\nu w_2+ i \langle A_2,\nu\rangle_g w_2)|_{\Gamma}= (\p_\nu u_1+ i \langle A_1,\nu\rangle_g u_1)|_{\Gamma}.
\end{equation}
\end{prop}

\begin{proof}
Let $u_1,u_2\in H^1(M^0)$ be solutions to \eqref{eq_10_1}.  As $C_{A_1,q_1}^\Gamma=C_{A_2,q_2}^\Gamma$, there is $w_2\in H^1(M^0)$ solving \eqref{eq_10_2} and satisfying \eqref{eq_10_3}. Using \eqref{eq_10_1}, \eqref{eq_10_2}, and \eqref{1.1}, we obtain that 
\begin{equation}
\label{eq_10_4}
\begin{aligned}
&L_{A_2,q_2}(w_2-u_1)\\
&= id^*((A_1-A_2)u_1)-i \langle A_1-A_2, du_1\rangle_g+ (\langle A_1, A_1\rangle_g -\langle A_2, A_2\rangle_g+q_1-q_2)u_1 \quad \text{in} \quad \mathcal{D}'(M^0). 
\end{aligned}
\end{equation}
Now it follows from \eqref{eq_10_1}, \eqref{eq_10_2}, and \eqref{eq_10_3} that $w_2-u_1\in H_0^1(M^0)$, $\Delta_g(w_2-u_1)\in L^2(M)$, and therefore, by the boundary elliptic regularity, $w_2-u_1\in H^2(M^0)$. Hence, $\p_\nu(w_2-u_1)|_{\p M}\in H^{\frac{1}{2}}(\p M)$. 

Multiplying \eqref{eq_10_4} by $\overline{u_2}$, using the magnetic Green formula \eqref{eq_4_Green_1}, \eqref{eq_10_1} and \eqref{eq_10_3}, we get 
\begin{equation}
\label{eq_10_5}
\begin{aligned}
\int_M (id^*((A_1-A_2)u_1)-i \langle A_1-A_2, du_1\rangle_g) \overline{u_2}dV_g+ \int_M (\langle A_1, A_1\rangle_g -\langle A_2, A_2\rangle_g+q_1-q_2)u_1\overline{u_2}dV_g\\
=-\int_{\p M}\p_\nu (w_2-u_1)\overline{u_2}dS_g.
\end{aligned}
\end{equation}
Using that 
\[
\int_M d^*((A_1-A_2)u_1)\overline{u_2}dV_g=\int_M \langle  (A_1-A_2) u_1,d\overline{u_2}\rangle_g dV_g-  
\int_{\p M} \langle  A_1-A_2,\nu\rangle_g u_1\overline{u_2}dS_g,
\]
we obtain from \eqref{eq_10_5} that 
\begin{equation}
\label{eq_10_6}
\begin{aligned}
\int_M i \langle A_1-A_2, u_1d\overline{u_2}-\overline{u_2}du_1\rangle_g dV_g+ \int_M (\langle A_1,A_1\rangle_g -\langle A_2, A_2\rangle_g+q_1-q_2)u_1\overline{u_2}dV_g\\
=-\int_{\p M}\p_\nu (w_2-u_1)\overline{u_2}dS_g+ i \int_{\p M} \langle  A_1-A_2,\nu\rangle_g u_1\overline{u_2}dS_g. 
\end{aligned}
\end{equation}
Now it follows from \eqref{eq_10_3} that 
\begin{equation}
\label{eq_10_6_2}
\p_\nu (w_2-u_1)|_{\Gamma}=i  \langle  A_1-A_2,\nu\rangle_gu_1|_{\Gamma}.
\end{equation} 
This together with \eqref{eq_10_6} shows \eqref{eq_10_6_prop}. 
\end{proof}

Let us now rewrite the integral identity of Proposition \ref{prop_integral_identity} in the following form, 
\begin{equation}
\label{eq_10_6_prop_rewritten}
\begin{aligned}
\int_M i \langle A_1-A_2, u_1d u_2- u_2du_1\rangle_g dV_g+ \int_M (\langle A_1, A_1\rangle_g -\langle A_2,A_2\rangle_g+q_1-q_2)u_1 u_2dV_g\\
=-\int_{\p M\setminus \Gamma}\p_\nu (w_2-u_1)u_2dS_g+ i \int_{\p M\setminus\Gamma} \langle  A_1-A_2,\nu\rangle_g u_1u_2dS_g,
\end{aligned}
\end{equation}
for $u_1,u_2\in H^1(M^0)$ satisfying 
\begin{equation}
\label{eq_10_1_rewritten}
L_{A_1,q_1} u_1=0, \quad L_{-A_2, q_2}u_2=0, \quad \text{in}\quad \mathcal{D}'(M^0),
\end{equation}
and $w_2\in H^1(M^0)$ satisfying \eqref{eq_10_2} and \eqref{eq_10_3}.

Next we shall test the integral identity \eqref{eq_10_6_prop_rewritten} against CGO solutions to equations \eqref{eq_10_1_rewritten}. By Proposition \ref{prop3.1}, for all  $h>0$ small enough, there are solutions $u_1, u_2 \in H^1(M^0)$ to \eqref{eq_10_1_rewritten}
 of the form 
 \begin{equation}
 \label{eq_10_7}
 u_1=e^\frac{\rho}{h}(a_1+r_1), \quad u_2=e^\frac{-\rho}{h}(a_2+r_2), 
 \end{equation}
  where $\rho = x_1 + i r$,  
 \begin{equation}
 \label{eq_10_8}
  a_1 = |g|^{-\frac{1}{4}}c^\frac{1}{2}e^{i\Phi^{(1)}_h}, \quad  a_2 = |g|^{-\frac{1}{4}}c^\frac{1}{2}e^{i\Phi^{(2)}_h}a_0(x_1,r)b(\theta), 
  \end{equation}
  $a_0$ is a non-vanishing holomorphic function so that $(\p_{x_1}+i\p_r)a_0 =0$,  $b(\theta)$ is smooth. Here $\Phi^{(1)}_h,\Phi^{(2)}_h \in C^\infty(M)$ are such that 
\begin{equation}
 \label{eq_10_9}
 \begin{aligned}
      \|a_j\|_{L^\infty(M)} = \mathcal{O}(1), \hspace{5mm} \|\nabla_g a_j\|_{L^\infty(M)} = \mathcal{O}(h^{-1/2}), \hspace{5mm} \|\Delta_g a_j\|_{L^\infty(M)} = \mathcal{O}(h^{-1}), \\
 \|a_j\|_{L^2(M)}=\mathcal{O}(1), \quad   \|\nabla_g a_j\|_{L^2(M)}=\mathcal{O}(1),\quad  \|\Delta_g a_j\|_{L^2} = o(h^{-1/2}).
  \end{aligned}
\end{equation}
as $h\to 0$, $j=1,2$, and 
\begin{equation}
 \label{eq_10_10}
 \|\Phi^{(j)}_h - \Phi^{(j)}\|_{L^n(M)} = o(h^{1/2}),
\end{equation}
as $h\to 0$, 
where 
\begin{equation}
 \label{eq_10_10_1}
 \Phi^{(1)}(x_1,r,\theta) = -\frac{1}{2} \frac{1}{\pi(x_1+ir)}* ((A_1)_{x_1}+i(A_1)_r), \quad 
 \Phi^{(2)}(x_1,r,\theta) = \frac{1}{2} \frac{1}{\pi(x_1+ir)}* ((A_2)_{x_1}+i(A_2)_r)  
 \end{equation}
 with $A_j=(A_j)_{x_1}dx_1+(A_j)_rdr+(A_j)_\theta d\theta$, $j=1,2$. Furthermore, the remainders $r_j$ satisfy 
\begin{equation}
 \label{eq_10_11}
 \|r_j\|_{H^1_{scl}(M^0)} = o(h^{1/2}),
\end{equation} 
as $h\to 0$, $j=1,2$. 

Let us set 
\[
J_h:=\int_{\p M\setminus \Gamma}\p_\nu (w_2-u_1) u_2dS_g,
\]
with $u_1,u_2$ being CGO solutions given by \eqref{eq_10_7}. The first step in the recovery of the magnetic field is the following result. 
\begin{prop}
\label{prop_10_boundary_terms}
We have
\begin{equation}
 \label{eq_10_12}
h |J_h|\to 0,
\end{equation} 
as $h\to 0$. 
\end{prop}

\begin{proof}
To prove this result,  we shall rely on boundary Carleman estimates of Corollary \ref{cor_prop2.2}. To that end, first recalling that $\Gamma$ is an open neighborhood of $F$, given by \eqref{eq_int_p_M-}, and letting $\varphi(x)=x_1$,  we see that there is $\varepsilon>0$ such that 
\[
F\subset \p M^\varphi_{-,\varepsilon}:=\{x\in \p M: \p_\nu \varphi(x)\le \varepsilon\}\subset \Gamma. 
\]
We get 
\begin{equation}
 \label{eq_10_13}
\begin{aligned}
|J_h|&\le \int_{\p M\setminus \Gamma}|\p_\nu (w_2-u_1) u_2|dS_g\le  \int_{\p M\setminus \p M^\varphi_{-,\varepsilon}}\frac{\sqrt{\varepsilon}}{\sqrt{\varepsilon}}|\p_\nu (w_2-u_1) u_2|dS_g\\
&\le \frac{1}{\sqrt{\varepsilon}} \int_{ \p M^\varphi_+}|\sqrt{\p_\nu \varphi} \p_\nu (w_2-u_1) u_2|dS_g,
\end{aligned}
\end{equation}
where $\p M^\varphi_+=\{x\in \p M: \p_\nu \varphi(x)\ge 0\}$.   Substituting $u_2$,  given by \eqref{eq_10_7}, in 
\eqref{eq_10_13}, and using the Cauchy--Schwarz inequality,  we obtain that 
\begin{equation}
 \label{eq_10_14}
\begin{aligned}
|J_h|&\le \frac{1}{\sqrt{\varepsilon}} \int_{ \p M^\varphi_+}|\sqrt{\p_\nu \varphi} \p_\nu (w_2-u_1)e^{-\frac{\varphi}{h}}(a_2+r_2)|dS_g\\
&\le \mathcal{O}(1)\|\sqrt{\p_\nu\varphi}e^{-\frac{\varphi}{h}}\p_\nu (w_2-u_1)\|_{L^2(\p M^\varphi_+)}\|a_2+r_2\|_{L^2(\p M)}.
\end{aligned}
\end{equation}
We shall proceed to bound the right hand side of  \eqref{eq_10_14}. In doing so we need boundary Carleman estimates of Corollary \ref{cor_prop2.2} with $\varphi(x)=-x_1$,   which read as follows, for $v\in H^2(M^0)\cap H_0^1(M^0)$, 
\begin{equation}
 \label{eq_10_15}
 \|\sqrt{\p_\nu\varphi}e^{-\frac{\varphi}{h}}\p_\nu v\|_{L^2(\partial M^\varphi_+)}\le \mathcal{O}(\sqrt{h})\|e^{-\frac{\varphi}{h}} L_{A_2,q_2} v\|_{L^2(M)}+ \mathcal{O}(1)\|\sqrt{-\p_\nu\varphi}e^{-\frac{\varphi}{h}}\p_\nu v\|_{L^2(\p M_-^\varphi)}.
\end{equation}
  Using \eqref{eq_10_15} and \eqref{eq_10_2}, we get from \eqref{eq_10_14} that 
\begin{equation}
 \label{eq_10_16}
\begin{aligned}
|J_h|\le  \big(\mathcal{O}(\sqrt{h})\|e^{-\frac{\varphi}{h}} L_{A_2,q_2} (w_2-u_1)\|_{L^2(M)}+\mathcal{O}(1)\|\sqrt{-\p_\nu\varphi}e^{-\frac{\varphi}{h}}\p_\nu (w_2-u_1)\|_{L^2(\partial M_-^\varphi)}\big)\\
\|a_2+r_2\|_{L^2(\p M)}\\
\le \big(\mathcal{O}(\sqrt{h})\|e^{-\frac{\varphi}{h}} L_{A_2,q_2} u_1\|_{L^2(M)}+\mathcal{O}(1)\|\sqrt{-\p_\nu\varphi}e^{-\frac{\varphi}{h}} \langle A_1-A_2,\nu\rangle_g u_1\|_{L^2(\partial M_-^\varphi)}\big)\\
\|a_2+r_2\|_{L^2(\p M)}.
\end{aligned}
\end{equation}
We used $F=\partial M_-^\varphi$ and \eqref{eq_10_6_2} in the last inequality. 

We shall first proceed to bound the second term in the last inequality in \eqref{eq_10_16}. To that end, we need the following general estimate, 
\begin{equation}
 \label{eq_10_17}
\|v\|_{H^{1/2}_{scl}(\p M)}\le \mathcal{O}(h^{-1/2})\|v\|_{H^1_{scl}(M^0)}, \quad v\in H^1(M^0),
\end{equation}
see \cite{Sjostrand_2014}. The estimates \eqref{eq_10_11} and \eqref{eq_10_17} imply that 
\begin{equation}
\label{eq_10_18}
\|r_j\|_{H^{1/2}_{scl}(\p M)}=o(1), 
\end{equation}
and therefore, 
\begin{equation}
\label{eq_10_18_1} 
\|r_j\|_{L^2(\p M)}=o(1), 
\end{equation}
as $h\to 0$, $j=1,2$. Furthermore, by the semiclassical version of Sobolev's embedding, see \cite{BurqDosSantosKrup_2018}*{Lemma 2.5},  and \eqref{eq_10_18}, we get 
\begin{equation}
\label{eq_10_19}
\|r_j\|_{L^{\frac{2(n-1)}{n-2}}(\p M)}\le \mathcal{O}(h^{-1/2}) \|r_j\|_{H^{1/2}_{scl}(\p M)}=o(h^{-1/2}), 
\end{equation}
as $h\to 0$, $j=1,2$.

Now using \eqref{eq_10_7}, \eqref{eq_10_9}, \eqref{eq_10_19},  H\"older's inequality, and Proposition \ref{prop_trace_magnetic_2},  we obtain that 
\begin{equation}
 \label{eq_10_20}
\begin{aligned}
&\|\sqrt{-\p_\nu\varphi}e^{-\frac{\varphi}{h}} \langle A_1-A_2,\nu\rangle_g u_1\|_{L^2(\partial M_-^\varphi)}\le \mathcal{O}(1)\|\langle A_1-A_2,\nu\rangle_g(a_1+r_1)\|_{L^2(\p M)}\\
&\le \mathcal{O}(1)\big(\|\langle A_1-A_2,\nu\rangle_g\|_{L^2(\p M)}\|a_1\|_{L^\infty(M)}+ \|\langle A_1-A_2,\nu\rangle_g\|_{L^{2(n-1)}(\p M)}\|r_1\|_{L^{\frac{2(n-1)}{n-2}}(\p M)}\big)=o(h^{-1/2}),
\end{aligned}
\end{equation}
as $h\to 0$. 

Let us now bound the first term in the last inequality in \eqref{eq_10_16}.  To that end, using the semiclassical version of Sobolev's embedding, see \cite{BurqDosSantosKrup_2018}*{Lemma 2.5}, and \eqref{eq_10_11}
\begin{equation}
 \label{eq_10_21_rem}
\|r_j\|_{L^{\frac{2n}{n-2}}(M)}\le \mathcal{O}(h^{-1})\|r_j\|_{H^1_{scl}(M^0)}=o(h^{-1/2}),
\end{equation}
as $h\to 0$, $j=1,2$. Now in view of  \eqref{eq_10_1_rewritten} and \eqref{1.1}, we write
\begin{equation}
 \label{eq_10_21}
-e^{-\frac{\varphi}{h}} L_{A_2,q_2} u_1=e^{-\frac{\varphi}{h}} \bigg[ i d^*(A_1-A_2)u_1-2i\langle A_1-A_2, du_1\rangle_g +\big(\langle A_1,A_1\rangle_g -\langle A_2,A_2\rangle_g+q_1-q_2\big)u_1\bigg].
\end{equation}
Using \eqref{eq_10_7}, H\"older's inequality, \eqref{eq_10_9},  \eqref{eq_10_21_rem},  we get
\begin{equation}
 \label{eq_10_22}
\begin{aligned}
 &\mathcal{O}(\sqrt{h})\|e^{-\frac{\varphi}{h}} (i d^*(A_1-A_2)+ q_1-q_2)u_1\|_{L^2(M)}\\
& \le \mathcal{O}(\sqrt{h})\|i d^*(A_1-A_2)+ q_1-q_2\|_{L^n(M)}\|a_1+r_1\|_{L^{\frac{2n}{n-2}}(M)}=o(1),
\end{aligned}
\end{equation}
as $h\to 0$. Using \eqref{eq_10_9}, \eqref{eq_10_11}, we obtain that 
\begin{equation}
 \label{eq_10_23}
\begin{aligned}
 \mathcal{O}(\sqrt{h})\|e^{-\frac{\varphi}{h}}\langle A_1-A_2, du_1\rangle_g \|_{L^2(M)}\le  \mathcal{O}(h^{-1/2})\|\langle A_1-A_2, d\rho \rangle_g  \|_{L^\infty(M)}\|a_1+r_1\|_{L^2(M)}\\
 +\mathcal{O}(\sqrt{h})\|A_1-A_2\|_{L^\infty(M)}\|da_1+dr_1\|_{L^2(M)}=\mathcal{O}(h^{-1/2}),
\end{aligned}
\end{equation}
as $h\to 0$. We also have
\begin{equation}
 \label{eq_10_24}
\begin{aligned}
& \mathcal{O}(\sqrt{h})\|e^{-\frac{\varphi}{h}}(\langle A_1,A_1\rangle_g -\langle A_2,A_2\rangle_g)u_1\|_{L^2(M)}\\
& \le 
 \mathcal{O}(\sqrt{h})\|\langle A_1,A_1\rangle_g -\langle A_2,A_2\rangle_g\|_{L^\infty(M)}\|a_1+r_1\|_{L^2(M)}=\mathcal{O}(\sqrt{h}),
\end{aligned}
\end{equation}
as $h\to 0$. Combining the estimates \eqref{eq_10_22}, \eqref{eq_10_23}, \eqref{eq_10_24}, in view of \eqref{eq_10_21},   
we see that 
\begin{equation}
 \label{eq_10_25}
\begin{aligned}
 &\mathcal{O}(\sqrt{h})\|e^{-\frac{\varphi}{h}} L_{A_2,q_2} u_1 \|_{L^2(M)}=\mathcal{O}(h^{-1/2}),
\end{aligned}
\end{equation}
as $h\to 0$. Now by \eqref{eq_10_18_1} and \eqref{eq_10_9}, we also have 
\begin{equation}
 \label{eq_10_26}
\|a_2+r_2\|_{L^2(\p M)}\le \mathcal{O}(1)\|a_2\|_{L^\infty(\p M)}+\|r_2\|_{L^2(\p M)}\le \mathcal{O}(1), 
\end{equation}
as $h\to 0$. 

Using \eqref{eq_10_20}, \eqref{eq_10_25}, and \eqref{eq_10_26}, we obtain from \eqref{eq_10_16} that 
\[
|J_h|=\mathcal{O}(h^{-1/2}),
\]
as $h\to 0$. This shows \eqref{eq_10_12}. 
\end{proof}

Our next step in recovering the magnetic field is the following result. 
\begin{prop}
\label{prop_10_boundary_terms_magnet}
Let $u_1,u_2$ be CGO solutions given by \eqref{eq_10_7}. The we have
\begin{equation}
 \label{eq_10_27}
h\bigg|\int_{\p M\setminus\Gamma} \langle  A_1-A_2,\nu\rangle_g u_1 u_2dS_g\bigg|\to 0,
\end{equation}
as $h\to 0$. 
\end{prop}
\begin{proof}
By H\"older's inequality, as in \eqref{eq_4_trace_A_bound}, using \eqref{eq_10_19}, and \eqref{eq_10_9},  we get 
\begin{align*}
\bigg|\int_{\p M} \langle A_1-A_2,\nu\rangle_g u_1 u_2dS_g\bigg|\le \|\langle A_1-A_2,\nu\rangle_g\|_{L^{n-1}(\p M)}\|a_1+r_1\|_{L^{\frac{2(n-1)}{n-2}}(\p M)}\|a_2+r_2\|_{L^{\frac{2(n-1)}{n-2}}(\p M)}\\
\le \mathcal{O}(1)\big(\|a_1\|_{L^\infty(M)}+\|r_1\|_{L^{\frac{2(n-1)}{n-2}}(\p M)}\big)
\big(\|a_2\|_{L^\infty(M)}+\|r_2\|_{L^{\frac{2(n-1)}{n-2}}(\p M)}\big)=o(h^{-1}), 
\end{align*}
as $h\to 0$. This shows \eqref{eq_10_27}. 
\end{proof}

Using Propositions \ref{prop_10_boundary_terms},  \ref{prop_10_boundary_terms_magnet}, we obtain from \eqref{eq_10_6_prop_rewritten} that for $u_1,u_2$ being CGO solutions given by \eqref{eq_10_7}, 
\begin{equation}
 \label{eq_10_28}
\begin{aligned}
h\int_M i \langle A_1-A_2, u_1d u_2- u_2du_1\rangle_g dV_g+ h\int_M (\langle A_1,A_1\rangle_g -\langle A_2,A_2\rangle_g+q_1-q_2)u_1 u_2 dV_g\\
=o(1),
\end{aligned}
\end{equation}
as $h\to 0$. 

The next step in recovering the magnetic field is the following result.
\begin{prop}
\label{prop-idenity_A1-A2}
The equality \eqref{eq_10_27} implies that 
\begin{equation}
 \label{eq_10_29}
 \int_M \langle A_1-A_2, d\rho \rangle_g  |g|^{-1/2}c e^{i\Phi} a_0 bdV_g=0,
\end{equation}
where $\Phi=\Phi^{(1)}+\Phi^{(2)}\in L^\infty(M)$ with  $\Phi^{(j)}$  given by \eqref{eq_10_10_1}.
\end{prop}

\begin{proof}
First using H\"older's inequality, and \eqref{eq_10_21_rem}, \eqref{eq_10_11}, \eqref{eq_10_9}, we get 
\begin{equation}
 \label{eq_10_30}
\begin{aligned}
\bigg| h\int_M & (\langle A_1,A_1\rangle_g -\langle A_2,A_2\rangle_g+q_1-q_2)u_1 u_2dV_g\bigg|\\
&\le
\mathcal{O}(h)\|\langle A_1,A_1\rangle_g -\langle A_2,A_2\rangle_g\|_{L^\infty(M)}\|a_1+r_1\|_{L^2(M)}\|a_2+r_2\|_{L^2(M)}\\
&+\mathcal{O}(h)\|q_1-q_2\|_{L^n(M)}\|a_1+r_1\|_{L^\frac{2n}{n-2}(M)}\|a_2+r_2\|_{L^2(M)}\le \mathcal{O}(h)+o(h^{1/2}),
\end{aligned}
\end{equation}
as $h\to 0$. Thus, in view of \eqref{eq_10_30},  it follows from \eqref{eq_10_28} that 
\begin{equation}
 \label{eq_10_31}
h\int_M  \langle A_1-A_2, u_1d u_2- u_2du_1\rangle_g dV_g
=o(1),
\end{equation}
as $h\to 0$. Here $u_1,u_2$ are CGO solutions given by \eqref{eq_10_7}.  Using \eqref{eq_10_7}, we have
\begin{equation}
 \label{eq_10_32}
 u_1d u_2- u_2du_1=-\frac{2d\rho}{h}(a_1a_2+a_1r_2+a_2r_1+r_1r_2)+(a_1+r_1)(da_2+dr_2)-(a_2+r_2)(da_1+dr_1).
\end{equation}
Using \eqref{eq_10_9}, \eqref{eq_10_11},  we get
\begin{equation}
 \label{eq_10_33}
\bigg|h\int_M  \langle A_1-A_2, (a_1+r_1)(da_2+dr_2) \rangle_g dV_g\bigg|\le \mathcal{O}(h)\|A_1-A_2\|_{L^\infty}\|a_1+r_1\|_{L^2}\|da_2+dr_2\|_{L^2}=o(h^{1/2}),
\end{equation}
as $h\to 0$. Similarly, we have 
\begin{equation}
 \label{eq_10_34}
\bigg|h\int_M  \langle A_1-A_2, (a_2+r_2)(da_1+dr_1)\rangle_g dV_g\bigg|=o(h^{1/2}),
\end{equation}
as $h\to 0$. Using \eqref{eq_10_9}, \eqref{eq_10_11},  we also get
\begin{equation}
 \label{eq_10_35}
\begin{aligned}
\bigg|\int_M  &\langle A_1-A_2, d\rho \rangle_g (a_1r_2+a_2r_1+r_1r_2)dV_g\bigg|\\
&\le  \| \langle A_1-A_2, d\rho \rangle_g \|_{L^\infty}(\|a_1\|_{L^2}\|r_2\|_{L^2}+\|a_2\|_{L^2}\|r_1\|_{L^2}+ \|r_1\|_{L^2}\|r_2\|_{L^2})=o(h^{1/2}),
\end{aligned}
\end{equation}
as $h\to 0$.  Thus, it follows from \eqref{eq_10_31}, \eqref{eq_10_32}, \eqref{eq_10_33}, \eqref{eq_10_34}, and \eqref{eq_10_35} that 
\begin{equation}
 \label{eq_10_36}
\int_M  \langle A_1-A_2, d\rho \rangle_g a_1 a_2dV_g=o(1),
\end{equation}
as $h\to 0$.  Using \eqref{eq_10_8}, we obtain from \eqref{eq_10_36} that 
\begin{equation}
 \label{eq_10_37}
\lim_{h\to 0}\int_M  \langle A_1-A_2, d\rho \rangle_g |g|^{-1/2}c e^{i(\Phi_h^{1}+\Phi_h^{(2)})}a_0b dV_g=0. 
\end{equation}
To show that \eqref{eq_10_37} gives the claim \eqref{eq_10_29}, we proceed as in 
\cite{2018}*{Section 4, page 545}. Indeed, using $|e^z-e^w| \leq |z-w| e^{\max\{\text{Re}\,(z),\text{Re}\,(w)\}}$, the fact that  $\Phi^{(j)}_h,\Phi_j \in L^\infty(M)$ and $\|\Phi^{(j)}_h\|_{L^\infty(M)} =\mathcal{O}(1)$ uniformly in $h$, cf. \eqref{3.15}, the embedding $L^n(M)\subset L^2(M)$, and \eqref{eq_10_10}, we get 
\begin{align*}
    \bigg|\int_M \langle A_1-A_2, d\rho\rangle_g |g|^{-1/2}c (&e^{i(\Phi^{(1)}_h+\Phi^{(2)}_h)}-e^{i\Phi})a_0 b dV_g\bigg| \le O(1) \|e^{i(\Phi^{(1)}_h+\Phi^{(2)}_h)}-e^{i\Phi}\|_{L^2(M)}\\
    &\leq O(1) \|\Phi^{(1)}_h+\Phi^{(2)}_h - \Phi_1-\Phi_2\|_{L^2(M)} =o(h^{1/2}), 
\end{align*}
as $h\to 0$. This complete the proof of \eqref{eq_10_37}. 
\end{proof}

Proposition \ref{prop-idenity_A1-A2} gives us exactly the same integral identity for $A_1-A_2$ as that in \cite{2018}*{Section 4, page 545}. Proceeding as that work, using only that $A_1,A_2\in L^\infty(M, T^*M)$, we conclude that $dA_1=dA_2$ in $M$.

\section{Recovering the electric potential}

\label{sec_recovery_electric}

Since $M$ is simply connected, by the Poincar\'e lemma for currents, see \cite{Rham}, we conclude that there is $\phi \in \mathcal{D}'(M)$ such that $d\phi = A_1-A_2\in (W^{1,n}\cap L^\infty) (M^0,T^*M^0)$. It follows from \cite{Hormander}*{Theorem 4.5.11} that $\phi\in W^{1,\infty}(M^0)$. Since $A_1 = A_2$ in $L^2(\p M)$, we have that 
$d(\phi|_{\p M})=0$ in $\mathcal{D}'(\p M)$. By \cite{Hormander}*{Theorem 3.1.4'} and the fact that $\p M$ is connected, 
$\phi$ is constant along $\p M$. Modifying $\phi$ by a constant,  we may assume that $\phi =0$ on $\partial M$. Let us set 
\[
C_{A_j,q_j}=\{(u|_{\p M}, (\p_\nu u+i\langle A, \nu\rangle_g u)|_{\p M }: u\in H^1(M^0) \text{ such that }L_{A_j,q_j}u=0 \text{ in } M^0\},
\]
$j=1,2$,  for the set of full Cauchy data. Recalling  \cite{2018}*{Lemma 4.1}, we get 
\[
C_{A_2,q_2}=C_{A_2+d\phi, q_2}= C_{A_1,q_2},
\]
and therefore, 
\begin{equation}
\label{eq_11_1}
C_{A_1,q_2}^\Gamma= C_{A_2,q_2}^\Gamma=C_{A_1,q_1}^\Gamma. 
\end{equation}
The equality \eqref{eq_11_1} gives us the following integral identity, see \eqref{eq_10_6_prop_rewritten}, 
\begin{equation}
\label{eq_11_2}
\begin{aligned}
\int_M (q_1-q_2)u_1 u_2dV_g
=-\int_{\p M\setminus \Gamma}\p_\nu (w_2-u_1)u_2dS_g,
\end{aligned}
\end{equation}
for $u_1,u_2\in H^1(M^0)$ satisfying 
\begin{equation}
\label{eq_11_3}
L_{A_1,q_1} u_1=0, \quad L_{-A_1, q_2}u_2=0, \quad \text{in}\quad \mathcal{D}'(M^0),
\end{equation}
and $w_2\in H^1(M^0)$ satisfying 
\begin{equation}
\label{eq_11_4}
 L_{A_1,q_2} w_2=0 \quad \text{in}\quad \mathcal{D}'(M^0),
\end{equation}
and 
\begin{equation}
\label{eq_11_5}
w_2|_{\p M}= u_1|_{\p M}, \quad \p_\nu w_2|_{\Gamma}= \p_\nu u_1|_{\Gamma}. 
\end{equation}

Next we shall test the integral identity \eqref{eq_11_2} against CGO solutions to equations \eqref{eq_11_3}. By Proposition \ref{prop3.1}, for all  $h>0$ small enough, there are solutions $u_1, u_2 \in H^1(M^0)$ to \eqref{eq_11_3}
 of the form \eqref{eq_10_7}, \eqref{eq_10_8}, satisfying the bounds \eqref{eq_10_9}, \eqref{eq_10_11}. 
 Let us set 
\[
I_h:=\int_{\p M\setminus \Gamma}\p_\nu (w_2-u_1) u_2dS_g,
\]
with $u_1,u_2$ being CGO solutions given by \eqref{eq_10_7}. The first step in recovering the electric potential is the following result. 
\begin{prop}
\label{prop_11_boundary_terms}
We have
\begin{equation}
 \label{eq_11_6}
|I_h|\to 0,
\end{equation} 
as $h\to 0$. 
\end{prop}

\begin{proof}
First proceeding as in the proof of Proposition \ref{prop_10_boundary_terms}, we get 
\begin{equation}
 \label{eq_11_7}
\begin{aligned}
|I_h|\le  \big(\mathcal{O}(\sqrt{h})\|e^{-\frac{\varphi}{h}} L_{A_1,q_2} (w_2-u_1)\|_{L^2(M)}+\mathcal{O}(1)\|\sqrt{-\p_\nu\varphi}e^{-\frac{\varphi}{h}}\p_\nu (w_2-u_1)\|_{L^2(\partial M_-^\varphi)}\big)\\
\|a_2+r_2\|_{L^2(\p M)}\\
\le \mathcal{O}(\sqrt{h})\|e^{-\frac{\varphi}{h}} L_{A_1,q_2} u_1\|_{L^2(M)}
\|a_2+r_2\|_{L^2(\p M)},
\end{aligned}
\end{equation}
cf. \eqref{eq_10_16}. In the last inequality of \eqref{eq_11_7}, we used that $F= \partial M_-^\varphi$ and \eqref{eq_11_4},  \eqref{eq_11_5}.

Now in view of \eqref{1.1}, \eqref{eq_11_3}, and  \eqref{eq_10_7},  we write
\begin{equation}
 \label{eq_11_8}
|e^{-\frac{\varphi}{h}} L_{A_1,q_2} u_1|=|e^{-\frac{\varphi}{h}}(q_1-q_2)u_1|=|(q_1-q_2)(a_1+r_1)|.
\end{equation}
Using H\"older's inequality and \eqref{eq_10_9}, \eqref{eq_10_21_rem}, \eqref{eq_10_26},  we obtain from \eqref{eq_11_7} that 
\begin{equation}
 \label{eq_11_9}
\begin{aligned}
|I_h|\le  \mathcal{O}(\sqrt{h})\| q_1-q_2\|_{L^n(M)}\|a_1+r_1\|_{L^{\frac{2n}{n-2}}(M)}
\|a_2+r_2\|_{L^2(\p M)}=o(1),
\end{aligned}
\end{equation}
as $h\to 0$. This shows \eqref{eq_11_6}. 
\end{proof}
Combining \eqref{eq_11_2} with Proposition \ref{prop_11_boundary_terms}, we get 
\begin{equation}
\label{eq_11_10}
\int_M (q_1-q_2)u_1 u_2dV_g
=o(1),
\end{equation}
as $h\to 0$,  with $u_1,u_2$ being CGO solutions given by \eqref{eq_10_7}. Substituting $u_1,u_2$ given by  \eqref{eq_10_7} into \eqref{eq_11_10}, we obtain that 
\begin{equation}
\label{eq_11_11}
\int_M (q_1-q_2)(a_1a_2+a_1r_2+a_2r_1+r_1r_2)dV_g
=o(1),
\end{equation}
as $h\to 0$. Using H\"older's inequality, and \eqref{eq_10_9}, \eqref{eq_10_11}, \eqref{eq_10_21_rem}, we get
\begin{equation}
\label{eq_11_12}
\begin{aligned}
\bigg|\int_M &(q_1-q_2)(a_1r_2+a_2r_1+r_1r_2)dV_g\bigg|
\le \|q_1-q_2\|_{L^n(M)}\\
&\bigg(\|a_1\|_{L^{\frac{2n}{n-2}}(M)}\|r_2\|_{L^2(M)} + \|a_2\|_{L^{\frac{2n}{n-2}}(M)}\|r_1\|_{L^2(M)}+ \|r_1\|_{L^{\frac{2n}{n-2}}(M)}\|r_2\|_{L^2(M)}\bigg)
=o(1),
\end{aligned}
\end{equation}
as $h\to 0$. Now in view of \eqref{eq_10_8}, we have 
\begin{equation}
\label{eq_11_13}
a_1a_2 = |g|^{-\frac{1}{2}}c e^{i(\Phi^{(1)}_h+\Phi^{(2)}_h )}a_0(x_1,r)b(\theta)= |g|^{-\frac{1}{2}}c a_0(x_1,r)b(\theta).
\end{equation}
Here we used that  $\Phi^{(1)}_h+\Phi^{(2)}_h=0$ which follows from the fact that 
 \begin{align*} 
  &\Phi^{(1)}_h(x_1,r,\theta) = -\frac{1}{2} \frac{1}{\pi(x_1+ir)}* ((A_{1,h})_{x_1}+i(A_{1,h})_r), \\
 &\Phi^{(2)}_h(x_1,r,\theta) = \frac{1}{2} \frac{1}{\pi(x_1+ir)}* ((A_{1,h})_{x_1}+i(A_{1,h})_r),
 \end{align*}
see  \eqref{3.14}. Choosing $a_0(x_1,r)=e^{i\lambda(x_1+ir)}$, $\lambda\in \R$, using \eqref{eq_11_12}, \eqref{eq_11_13}, and $dV_g=|g|^{1/2}dx_1drd\theta$,  we get from \eqref{eq_11_11}  that 
\begin{equation}
\label{eq_11_14}
\int_M (q_1-q_2) c b(\theta) e^{i\lambda(x_1+ir)} dx_1drd\theta=0.
\end{equation}
The integral identity \eqref{eq_11_14} is exactly the same as  that of \cite{unboundedPotential}*{Section 4, page 63}, where the case of potentials of class $L^{n/2}(M)\supset L^n(M)$ is considered. Proceeding as that work, we conclude that $q_1=q_2$ in $M$.  This concludes the proof of Theorem \ref{thm1.1}.

\section{Proof of Theorem \ref{thm1.2}}

\label{sec_Thm_1_2}

The proof is based on a well known reduction, which we recall following 
\cite{Kenig_Salo_2013}*{Section 3 b}, \cite{caseyRodriguez}*{Page 12}, for the convenience of the reader. 

Without loss of generality, we assume $x_0 = 0$ and that $\overline{\Omega} \subseteq \{ x \in \mathbb{R}^n : x_n >0\}$. Then  $\varphi(x) = \log|x|$. Using the notation of Theorem \ref{thm1.1}, let $M = \overline{\Omega}$ and let $g$ be the Euclidean metric on $\overline{\Omega}$. Consider the following change of coordinates
    \[
        y_1 = \log|x|,\ 
        y' = \frac{x}{|x|}. 
    \]
    Here $y'$ parametrizes $\mathbb{S}^{n-1}$. The Euclidean metric $g$ on $\overline{\Omega}$ is given by
    \begin{equation}
        \label{6.1}
        g=c(e \oplus g_0),
    \end{equation}
    where $c(y_1,y') = \exp(2y_1)$, $e$ is the Euclidean metric on $\mathbb{R}$, and $g_0 = g_{\mathbb{S}^{n-1}}$ is the standard metric on $\mathbb{S}^{n-1}$. To see why \eqref{6.1} is true, note that by \cite{grigoryan}*{Section 3.9, Equation 3.58} we have in polar coordinates,
 \[  
   g = dr^2 +r^2 g_0, \quad r=|x|.
 \]
    Since $r = e^{y_1}$ and so $dr = e^{y_1}dy_1$, we get
    \[
        g = (e^{y_1}dy_1)^2 +(e^{y_1})^2g_0= e^{2y_1}(dy_1^2+ g_0)= c(e \oplus g_0).
    \]
So finally we have
 \[  
   (\overline{\Omega},g) = (M,g) \subseteq (\mathbb{R}\times M_0^0, g),
\]
    where $(M_0,g_0) \subseteq (\{\theta \in \mathbb{S}^{n-1}: \theta_n >0\},g_0)$ is a closed cap, and thus a simple manifold.  The Euclidean magnetic Schr\"odinger operator given by  \eqref{1.1_eclid}, expressed in the $y$--coordinates takes the form \eqref{1.1}, with $g$  defined by \eqref{6.1}.  Since $\varphi(y_1,y')=y_1$ and $F=F(0)$ in view of 
 \eqref{eq_int_p_M-}, \eqref{eq_int_p_M-_eclid},  Theorem \ref{thm1.2} follows directly from Theorem \ref{thm1.1}.

\section{Application to inverse problems for  advection-diffusion equations: proof of Theorem \ref{thm1.3}}

\label{sec_Thm_1_3}

Let $X_l\in (W^{1,n}\cap L^\infty) (M^0,TM^0)$ and let $X_l^\flat = g_{jk}X_l^jdx^k$ be the corresponding 1-form, $l=1,2$. 
First a direct computation using \eqref{1.1} shows that 
\[
L_{X_l}=L_{A_l,q_l},
\]
where 
\[
A_l=\frac{iX_l^\flat}{2}\in (W^{1,n}\cap L^\infty) (M^0,T^*M^0), \quad q_l=  \frac{1}{4}\langle X_l, X_l \rangle_g - \frac{1}{2}{\rm{div}_g} (X_l)\in L^{n}(M,\R),
\]
$l=1,2$.  The fact that $\Lambda_{X_1}^\Gamma=\Lambda_{X_2}^\Gamma$ and the boundary reconstruction result of \cite{krupchykManifold}*{Appendix A} implies that $X_1|_{\Gamma}=X_2|_{\Gamma}$. Therefore, $C_{A_1,q_1}^\Gamma=C_{A_2,q_2}^\Gamma$. An application of Theorem \ref{thm1.1} gives that 
\[
dX_1^\flat=dX_2^\flat, \quad q_1=q_2,\quad \text{in}\quad M.
\]
Arguing as in the proof of Theorem \ref{thm1.1}, in particular, since $X_1|_{\p M} = X_2|_{\p M}$, we conclude that there exists $\phi\in W^{1,\infty}(M^0)$ such that 
\begin{equation}
\label{eq_20_1}
X_1-X_2=\nabla_g \phi,
\end{equation}
and 
\begin{equation}
\label{eq_20_2}
\phi|_{\p M}=0.
\end{equation}
 The fact that $q_1=q_2$ together with \eqref{eq_20_1}, \eqref{eq_20_2} implies that $\phi$ solves the following Dirichlet problem
 \[
\begin{cases}
\Delta_g(\phi)-V(\phi)=0& \text{in} \quad M^0,\\
\phi|_{\p M}=0. 
\end{cases}
\]
Here $V=X_2+\frac{1}{2}\langle \nabla_g \phi,\nabla_g \rangle_g\in L^\infty(M,TM)$.  Applying the maximum principle of \cite{Aubin}*{Chapter 3, Section 8.2}, we obtain that $\phi=0$ in $M$, and hence, $X_1=X_2$ in $M$. 

\appendix

\section{Regularization of Sobolev functions}

\label{sec_app}

In this appendix, we collect some useful estimates used in the main part of the paper. 

In doing so, we let $\Psi\in C^\infty_0(\R^n)$ be such that $\int_{\R^n} \Psi(x) dx=1$ and $0\le \Psi\le 1$. Let $\tau>0$ and let 
\[
\Psi_\tau(x)=\tau^{-n}\Psi\bigg(\frac{x}{\tau}\bigg), 
\]
be the usual mollifier. 

\begin{prop}
\label{propA.1}
Let $f\in W^{1,p}(\mathbb{R}^n)$,  $p\in [1,\infty)$, and assume that $\Psi$ is radial. Then $f* \Psi_\tau \in (C^\infty \cap W^{1.p})(\mathbb{R}^n)$ and 
\begin{equation}
\label{eq_app_10_1}
\|f*\Psi_\tau - f\|_{L^p(\mathbb{R}^n)} = o(\tau),
\end{equation}
as $\tau \to 0$. 
\end{prop}

\begin{proof} 
First, it is clear that $f* \Psi_\tau \in (C^\infty \cap W^{1.p})(\mathbb{R}^n)$. To show \eqref{eq_app_10_1}, we write
 \begin{align*}
     (f*\Psi_\tau)(x) - f(x) &= \int f(x-y)\Psi_\tau(y)dy-f(x)= \int f(x-\tau y)\Psi(y)dy - f(x) \\
     &= \int [f(x-\tau y)-f(x)]\Psi(y) dy = \int [f(x+\tau y)-f(x)]\Psi(y) dy \\
     &= \int_{\mathbb{R}^n} \bigg(\int_0^1 \frac{d}{dt}f(x+\tau t y)dt\bigg)\Psi(y) dy = \tau \int_{\mathbb{R}^n} \bigg(\int_0^1 \nabla f(x+\tau t y) \cdot y dt\bigg)\Psi(y) dy. 
 \end{align*}
 Since $\Psi$ is even, we have $\int_{\mathbb{R}^n} \Psi(y)y_j dy =0$ for $1\leq j \leq n$, and therefore $\int \nabla f(x) \cdot y \Psi(y)dy =0$. Using this last equality we get
\[
f*\Psi_\tau - f  = \tau \int_{\mathbb{R}^n} \bigg(\int_0^1 (\nabla f(x+\tau t y) -\nabla f(x)\bigg)dt \cdot y\Psi(y)dy.
\]
 Hence by Minkowski's inequality, we get 
 \[
 \|f*\Psi_\tau - f\|_{L^p(\mathbb{R}^n)} \leq \tau \int_{\mathbb{R}^n} |y| |\Psi(y)| \int_0^1 \|\nabla f(\cdot + \tau t y) - \nabla f(\cdot)\|_{L^p(\mathbb{R}^n)} dt dy=o(\tau), 
 \]
 as $\tau \to 0$. In the last inequality we have used that $\|\nabla f (\cdot+h)-\nabla f \|_{L^p(\R^n)}\to 0$ as $h\to 0$, since $\nabla f\in L^p(\R^n)$.  
 \end{proof}

\begin{prop}\label{propA.3}
Let $f\in (W^{1,p}\cap L^\infty)(\mathbb{R}^n), 1 \leq p < \infty$, and let $f_\tau=f* \Psi_\tau \in (C^\infty \cap W^{1.p}\cap L^\infty)(\mathbb{R}^n)$.
We have 
\begin{equation}
\label{eq_app_10_2}
\|f_\tau\|_{L^p}=\mathcal{O}(1), \quad \|\nabla f_\tau\|_{L^p}=\mathcal{O}(1), \quad \|\nabla^2f_\tau\|_{L^p}=o(\tau^{-1}),
\end{equation}
and 
\begin{equation}
\label{eq_app_10_3}
\|\nabla^k f_\tau \|_{L^\infty}=\mathcal{O}(\tau^{-k}), \quad k=0,1,2,\dots,
\end{equation}
as $\tau\to 0$. Here $\nabla^k f_\tau=\sum_{|\alpha|=k}\p^\alpha  f_\tau$.

\end{prop}

\begin{proof} The first two estimates in \eqref{eq_app_10_2} are clear in view of Young's inequality. The estimate \eqref{eq_app_10_3} follows from 
\[
\|\nabla^k f_\tau \|_{L^\infty}\le \|f\|_{L^\infty}\|\nabla^k  \Psi_\tau \|_{L^1}=\mathcal{O}(\tau^{-k}),
\]
$k=0,1,2,\dots$. Thus, we only need to prove the third estimate in \eqref{eq_app_10_2}. To that end, we first write for $|\alpha|=2$, 
\[
\p^\alpha f_\tau= \p_{x_k}f* \p_{x_j}\Psi_\tau, 
\]
for some $j,k$. Letting $g=\p_{x_k}f\in L^p$, we get 

 \begin{align*}
     (g*\partial_{x_j}\Psi)(x) &= \int g(y) (\partial_{x_j} \Psi_\tau)(x-y)dy=\tau^{-n-1}\int g(y)(\partial_{x_j}\Psi)\bigg(\frac{x-y}{\tau}\bigg)dy\\
     &= \tau^{-1}\int g(x-\tau y)(\partial_{y_j}\Psi)(y)dy=\tau^{-1} \int [g(x-\tau y)-g(x)]\partial_{y_j}\Psi(y)dy. 
 \end{align*} 
 Thus, by an  application of Minkowski's inequality, we obtain that 
 $$
 \|g*\partial_{x_j}\Psi_\tau\|_{L^p(\mathbb{R}^n)} \leq \tau^{-1} \int \underbrace{\|g(\cdot -\tau y)-g(\cdot)\|_{L^p}}_{\to 0 \text{ as } \tau \to 0} |\partial_{y_j}\Psi(y)|dy=o(\tau^{-1}),
 $$
 as $\tau\to 0$. This completes the proof of the third estimate in \eqref{eq_app_10_2}. 
\end{proof}

\section*{Acknowledgements}

The research of S.S. is partially supported by the National Science Foundation (DMS 2109199).


\end{document}